\documentclass [12pt]{amsart}
\usepackage{amssymb,amsmath,amscd,amsthm}
\begin{document}

\newtheorem{theorem}{Theorem}
\newtheorem{lemma}{Lemma}
\newtheorem{definition}{Definition}
\newtheorem{remark}{Remark}

\begin{center}
  {\Large {\bf Automorphisms of the semigroup
  of invertible matrices with nonnegative elements over commutative
partially ordered rings}}

\end{center}
 \vspace{.3truecm}

{\large  Bunina E.I., Semenov P.P.}

\bigskip
\bigskip

Suppose that  $R$ is an ordered ring, $G_n(R)$ is a subsemigroup of
$GL_n(R)$, consisting of all matrices with nonnegative elements. In
the paper~[1] A.V.~Mikhalev and M.A.~Shatalova described all
automorphisms of $G_n(R)$, where  $R$ is a linearly ordered
skewfield and $n\ge 2$. In~[3] E.I. Bunina and A.V.~Mikhalev found
all automorphisms of $G_n(R)$, if $R$ is an arbitrary linearly
ordered associative ring with $1/2$, $n\ge 3$. In this paper we
describe automorphisms of $G_n(R)$, if $R$ is a commutative
partially ordered ring, containing~$\mathbb{Q}$, $n\ge 3$.

\section{Necessary definitions and notions, formulation of the main theorem}\leavevmode

Let $R$ be an associative (commutative) ring with~$1$.

\begin{definition}
\emph{A ring   $R$ is called  \emph{partially ordered} if there is a
partial  order relation $\le$ on~$R$ satisfying the following
conditions}\emph{:}

\emph{1)} $\forall x,y,z\in R (x\le y \Rightarrow x+z\le y+z)$;

\emph{2)} $\forall x,y \in R (0\le x \land 0\le y \Rightarrow 0\le
xy)$.

\emph{We will consider such partially ordered rings that contain
$1/m \ge 0$ for all  $m\in \mathbb N$.}
\end{definition}

Elements $r\in R$ with $0\le r$ are called \emph{nonnegative}.

\begin{definition}
\emph{Let $R$ be a partially ordered ring. By $G_n(R)$ we denote the
subsemigroup of $GL_n(R)$, consisting of all matrices with
nonnegative elements.}
\end{definition}

The set of all invertible elements of~$R$ is denoted by~$R^*$. If
$1/2\in R$, then  $R^*$ is infinite, since it contains all $1/2^n$
for $n\in \mathbb N$. The set $R_+\cap R^*$ is denoted by~$R_+^*$.
If $1/2\in R$, it is also infinite.

\begin{definition}
\emph{Let $I=I_n$, $\Gamma_n(R)$ be the group consisting of all
invertible matrices from $G_n(R)$, $\Sigma_n$ be symmetric group of
order~$n$, $S_\sigma$ be the matrix of a substitution $\sigma\in
\Sigma_n$ \emph{(}i.\,e. the matrix $(\delta_{i\sigma(j)})$, where
$\delta_{i\sigma(j)}$ is the Kroneker symbol\emph{)}, $S_n=\{
S_\sigma|\sigma\in \Sigma_n\}$, $diag[d_1,\dots,d_n]$ be a diagonal
matrix with elements $d_1,\dots,d_n$ on the diagonal,
$d_1,\dots,d_n\in R_+^*$. By $D_n(R)$ we denote the group of all
invertible diagonal matrices from $G_n(R)$.}
\end{definition}

\begin{definition}
\emph{If  ${\mathcal A},{\mathcal B}$ are subsets of $G_n(R)$, then
$$
C_{\mathcal A}({\mathcal B})=\{ a\in {\mathcal A}|\forall b\in
{\mathcal B}\ (ab=ba)\}.
$$}
\end{definition}

Let $E_{ij}$ be a matrix unit.

\begin{definition}
\emph{By $B_{ij}(x)$ we denote the matrix $I+xE_{ij}$. ${\mathbf P}$
is a subsemigroup in $G_n(R)$, generated by all matrices $S_\sigma$
($\sigma\in \Sigma_n$), $B_{ij}(x)$ ($x\in R_+, i\ne j$) and
$diag[\alpha_1,\dots,\alpha_n]\in D_n(R)$.}
\end{definition}

\begin{definition}
\emph{Two matrices $A,B\in G_n(R)$ are called \emph{$\mathcal
P$-equivalent} (see~[1]), if there exist matrices $A_j\in G_n(R)$,
$j=0,\dots,k$, $A=A_0,B=A_k$, and matrices $P_i,\widetilde P_i,
Q_i,\widetilde Q_i\in \mathbf P$, $i=0,\dots , k-1$ such that
$P_iA_i\widetilde P_i=Q_iA_{i+1}\widetilde Q_i$.}
\end{definition}

\begin{definition}
\emph{By $GE_n^+(R)$ we denote a subsemigroup in $G_n(R)$, generated
by all matrices  $\mathcal P$-equivalent to matrices from~$\mathbf
P$.}
\end{definition}

Note that if (for example) $R$ is a linearly ordered field, then
$GE_n^+(R)=G_n(R)$.

\begin{definition}
\emph{If $G$ if some semigroup (for example,
 $G=R_+^*, G_n(R), GE_n^+(R)$),
then a homomorphism $\lambda(\cdot ): G\to G$ is called a
\emph{central homomorphism of}~$G$, if $\lambda(G)\subset Z(G)$. A
mapping $\Omega(\cdot): G\to G$ such that $\forall X\in G$
$$
\Omega (X)=\lambda(X)\cdot X,
$$
where $\lambda(\cdot)$ is a central homomorphism, is called a
\emph{central homotety}.}
\end{definition}

For example, if $R=\mathbb R$ (the field of real numbers), then a
homomorphism $\lambda(\cdot): G_n(\mathbb R)\to G_n(\mathbb R)$ such
that
 $\forall A\in G_n(\mathbb R)$ $\lambda(A)=|det\, A|\cdot I$, is a central homomorphism, and a mapping $\Omega(\cdot): G_n(\mathbb
R)\to G_n(\mathbb R)$ such that
 $\forall A\in G_n(\mathbb R)$
$\Omega(A)=|det\, A|\cdot A$, is a central homotety. Note that a
central homotety  $\Omega(\cdot)$ always is an endomorphism of a
semigroup~$G$: $\forall X,Y\in G$
$\Omega(X)\Omega(Y)=\lambda(X)X\cdot \lambda(Y)Y=
\lambda(X)\lambda(Y) X\cdot Y=\lambda(XY)XY=\Omega(XY).$

For every matrix  $M\in \Gamma_n(R)$ let $\Phi_M$ denote an
automorphism of $G_n(R)$ such that $\forall X\in G_n(R)$
$\Phi_M(X)=MXM^{-1}$.

For every  $y(\cdot)\in Aut (R_+)$ by $\Phi^y$ we denote an
automorphism of $G_n(R)$ such that $\forall X=(x_{ij})\in G_n(R)$
$\Phi^y(X)=\Phi^y((x_{ij}))=(y(x_{ij}))$.

The main result of this paper is the following

 {\bf Theorem.}
\emph{Let $\Phi$ be an automorphism of a semigroup $G_n(R)$, $n\ge
3$, $\mathbb{Q}\in R$. Then on the semigroup $GE_n^+(R)$
$\Phi=\Phi_M\Phi^c\Omega$, where $M\in \Gamma_n(R)$, $c(\cdot)\in
Aut(R_+)$, $\Omega(\cdot)$ is a central homotety of the semigroup
$GE_n^+(R)$.}

\section{Constructing of an automorphism $\Phi'$}\leavevmode

In this section we suppose that some automorphism $\Phi\in
Aut(G_n(R))$ is fixed, where $n\ge 3$, $\mathbb{Q}\subset R$, and by
this automorphism we construct a new automorphism $\Phi'\in
Aut(G_n(R))$ such that $\Phi'=\Phi_{M'}\Phi$ for some matrix $M'\in
\Gamma_n(R)$ and for all $\sigma\in \Sigma_n$ we have
$\Phi'(S_\sigma)= \alpha^{sgn\, \sigma}S_\sigma$, $\alpha^=1$.

\begin{lemma}
$\forall x,y \in R_+ (x+y=0 \Rightarrow x=0 \land y=0)$.
\end{lemma}
\begin{proof}
By definition of a partially ordered ring we obtain $0 \le y
\Rightarrow x \le x+y$, $x \le 0$ therefore, $x=0$, similarly $y=0$.

\end{proof}

\begin{lemma}\label{diag}
If $\Phi$ is an automorphism of $G_n(R)$, where $n\ge 3$,
$\mathbb{Q}\subset R$, then

\emph{1)} $\Phi(\Gamma_n(R))=\Gamma_n(R),$

\emph{2)} $\Phi(D_n(R))=D_n(R),$
\end{lemma}

\begin{proof}
1) Since $\Gamma_n(R)$ is the subgroup of all invertible matrices of
$G_n(R)$, then $\Phi(\Gamma_n(R))=\Gamma_n(R)$.

2) Consider the set $\mathcal F$ of all matrices $A\in \Gamma_n(R)$,
commuting with all matrices that are conjugate to~$A$.

Consider
$$
B=diag [\alpha_1,\dots,\alpha_n]\in D_n (R),
$$
Let
$$
A=\begin{pmatrix}
a_{11}& \dots & a_{1n}\\
\vdots& \ddots& \vdots\\
a_{n1}& \dots& a_{nn}
\end{pmatrix} \in \Gamma_n(R),  \quad
A^{-1}=\begin{pmatrix}
a_{11}'& \dots & a_{1n}'\\
\vdots& \ddots& \vdots\\
a_{n1}'& \dots& a_{nn}'
\end{pmatrix}.
$$
We have
$$
\sum_{k=1}^n a_{ik}'\cdot a_{kj} =0 \text{ для всех } i\ne j.
$$
Therefore, $a_{ik}'\cdot a_{kj}=0$ for all $i\ne j$ (by Lemma~1).
Then $A^{-1} B A$ is a diagonal matrix, so $D_n(R)\subset {\mathcal
F}$.

Suppose that there exists a matrix $C\in {\mathcal F}\setminus
D_n(R)$,
$$
C=\begin{pmatrix}
c_{11}& \dots & c_{1i}&\dots & c_{1j}&\dots & c_{1n}\\
\vdots& \ddots& \vdots& \vdots& \vdots& \vdots&\vdots\\
c_{i1}&\dots&c_{ii}&\dots &c_{ij}& \dots& c_{in}\\
\vdots& \vdots& \vdots& \ddots& \vdots& \vdots& \vdots\\
c_{j1}&\dots&c_{ji}& \dots& c_{jj}& \dots & c_{jn}\\
\vdots& \vdots&\vdots &\vdots &\vdots& \ddots& \vdots\\
 c_{n1}&\dots& c_{in}& \dots & c_{jn} &\dots& c_{nn}
\end{pmatrix}.
$$
Let us conjugate $C$ by the matrix $diag[d,1,\dots,1]\cdot S_{i,j}$
($i,j\ne 1$). The conjugate matrix ($C'$) has the form
$$
C=\begin{pmatrix}
c_{11}& \dots & c_{1j}d^{-1}&\dots & c_{1i}d^{-1}&\dots & c_{1n}d^{-1}\\
\vdots& \ddots& \vdots& \vdots& \vdots& \vdots&\vdots\\
c_{j1}d&\dots&c_{jj}&\dots &c_{ji}& \dots& c_{jn}\\
\vdots& \vdots& \vdots& \ddots& \vdots& \vdots& \vdots\\
c_{i1}d&\dots&c_{ij}& \dots& c_{ii}& \dots & c_{in}\\
\vdots& \vdots&\vdots &\vdots &\vdots& \ddots& \vdots\\
 c_{n1}d&\dots& c_{nj}& \dots & c_{nj} &\dots& c_{nn}
\end{pmatrix}.
$$

Using the condition $CC'=C'C$, we obtain
\begin{multline*}
c_{11}^2+c_{12}c_{21}d+\dots+c_{1i}c_{j1}d+\dots+c_{1j}c_{i1}d+\dots+c_{1n}c_{n1}d=\\
=c_{11}^2+c_{21}c_{12}d^{-1}+\dots+c_{i1}c_{1j}d^{-1}+\dots+c_{j1}c_{1i}d^{-1}+\dots+
c_{n1}c_{1n}.
\end{multline*}
Taking $d=2$ we obtain
$$
3\cdot(c_{12}c_{21}+\dots+c_{1i}c_{j1}+\dots+c_{1j}c_{i1}+\dots+c_{1n}c_{n1})=0.
$$
Therefore (by Lemma~1),
$$
c_{1i}c_{j1}=0 \text{ for all $i,j\ne 1$}.
$$
Similarly,
$$
c_{ki}c_{jk}=0\text{ for all $i,j\ne k$}.
$$
Let
$$
C^{-1}= \begin{pmatrix} \gamma_{11}&\dots &\gamma_{1n}\\
\vdots&\ddots&\vdots\\
\gamma_{n1}&\dots& \gamma_{nn}
\end{pmatrix}.
$$
Then
$$
0=\gamma_{11} c_{1k}+\gamma_{12}c_{2k}+\dots+\gamma_{1n}c_{nk},
$$
 thus
 $$
 \gamma_{11}c_{1k}=0, \quad k\ne 1.
 $$
We know
 $$
 1=\gamma_{11}c_{11}+\gamma_{12}c_{21}+\dots +\gamma_{1n}c_{n1}.
 $$
 Multiplying this equality by $c_{1k}$, we obtain
 $c_{1k}=0$. Similarly, $c_{ij}=0$ for $i\ne j$. Therefore,
 ${\mathcal F}=D_n(R)$, and so $\varphi(D_n(R))=D_n(R)$.
\end{proof}

\begin{lemma}\label{l1}
Let $\tau=(12)(34)\dots (2[n/2]-1, 2[n/2])\in S_n$. If $\Phi$ is an
automorphism of $G_n(R)$, $n\ge 3$, then there exists a matrix $M\in
\Gamma_n(R)$ such that $\Phi_M \Phi(S_\tau) =bS_\tau$, $b\in R_+^*$,
$b^2=1$.
\end{lemma}

\begin{proof}
Consider a matrix
$$
A=\begin{pmatrix} a_{11}& \dots &a_{1n}\\
\vdots& \ddots & \vdots \\
a_{n1} & \dots & a_{nn}
\end{pmatrix}
$$
such that $A^2=1$. It satisfies the conditions $a_{ik}a_{kj}=0$ for
all $i\ne j$. Moreover,
$$
1=a_{11}^2+a_{12}a_{21}+\dots +a_{1n}a_{n1}.
$$
Let $a_{1i}a_{i1}=e_i$, $i=1,\dots,n$. Then  $\{e_i\}$ is a system
of mutually orthogonal central idempotents of~$R$, with the sum~$1$,
i.\,e. $R=e_1R\oplus e_2R\oplus \dots\oplus e_nR$. Let us represent
$A$ as the sum $e_1A+\dots +e_nA$. Write $A_i=e_iA$ explicitly:

\begin{align*}
A_1&=\begin{pmatrix} a_{11} & 0& \dots & 0\\
0& e_1 a_{22} &\dots & e_1 a_{2n}\\
\vdots& \vdots& \ddots& \vdots\\
0& e_1 a_{n2} &\dots& e_1 a_{nn}
\end{pmatrix},\\
A_i&=\begin{pmatrix} 0& 0&           0  & a_{1i} &0 & 0\\
                    0& e_i a_{22} & \dots& 0& \dots& e_i a_{2n}\\
0& \vdots& *&  0&*& \vdots\\
a_{i1}& 0& 0&  0&0& 0\\
0& \vdots& *&  0& *& \vdots\\
0& e_i a_{n2}& \dots &0& \dots& e_i a_{nn}
\end{pmatrix},\quad i> 1.
\end{align*}

Conjugate the matrix $A$ by the matrix
$B=e_1B+\dots+e_nB=B_1+\dots+B_n$, where $B_1=e_1\cdot I$,
$B_i=e_i\cdot S_{(2i)}$. Then $A'=B^{-1}AB=e_1A'+\dots
e_nA'=A_1'+\dots + A_n'$, where $A_1'=A_1$,
$$
A_i'=\begin{pmatrix} 0& a_{1i}& 0&\dots& 0\\
a_{i1}& 0& 0& \dots& 0\\
0& 0& \alpha_{11}& \dots & \alpha_{1,n-2}\\
\vdots& \vdots&\vdots&\ddots& \vdots\\
0& 0& \alpha_{n-2,1}& \dots& \alpha_{n-2,n-2}
\end{pmatrix}, \quad i>1.
$$
Denote the matrix $\begin{pmatrix} \alpha_{11}&\dots &
\alpha_{1,n-2}\\
\vdots& \ddots& \vdots\\
\alpha_{n-2,1}& \dots& \alpha_{n-2,n-2} \end{pmatrix}$ (and in the
case
$i=1$ the matrix $\begin{pmatrix} e_1 a_{22}& \dots & e_1 a_{2n}\\
\vdots& \ddots& \vdots\\
e_1 a_{n2}& \dots& e_1 a_{nn} \end{pmatrix}$) by $\Lambda_i$. Note
that $\Lambda_i$ is a matrix over the ring $e_iR$ of  size $<n$ such
that $\Lambda_i^2=E$. Therefore we can repeat previous arguments for
the matrix $\Lambda_i$ and corresponding system of orthogonal
idempotents of $e_iR$. Finally we obtain a matrix~$\widetilde A$,
conjugate to the initial matrix~$A$, and consisting of diagonal
blocks $2\times 2$ and $1\times 1$.

Consequently every element of order two in the group $\Gamma_n(R)$
is conjugate to some matrix consisting of diagonal blocks  $2\times
2$ and $1\times 1$.

Consider the set

\begin{multline*}
 {\mathcal F}=\{ D\in D_n(R)\mid \\
 \mid \forall N\in \Gamma_n(R)
 (N^2=I\Rightarrow \exists C\in \Gamma_n(R) (D(CNC^{-1})=
 (CNC^{-1}D)))\}.
 \end{multline*}

This set consists of all  diagonal matrices~$D$ such that for every
element of the semigroup of order two there exists a matrix
conjugate to this element that commutes with~$D$.
 Then $D$ in some basis has the form $diag[\alpha_1,\alpha_1,\alpha_2,\alpha_2,\dots ,
 \alpha_k,\alpha_k,\alpha]$, if $n=2k+1$, or $diag
 [\alpha_1,\alpha_1,\dots ,\alpha_k,\alpha_k]$, if $n=2k$, ($*$)
 since
$N$  can be taken equal to  $S_\tau$.

 Then consider a set
 $$
 \Lambda=\{ D\in {\mathcal F}\mid \forall D'\in {\mathcal F}
 C_{\Gamma_n(R)} (D) \not\supset C_{\Gamma_n(R)}(D')\},
 $$
consisting of matrices from $\mathcal F$ with minimal centralizers.

Every matrix from ${\mathcal F}$ commutes (in a basis, where it has
the form ($*$)) with all matrices that are divided in this basis to
diagonal blocks $2\times 2$ (and, possibly, a block $1\times 1$ in
the end). Therefore, $\Lambda $ contains matrices with  centralizer
only of such matrices. From the other side, in every basis matrices
with given properties exist, for example these are matrices
$diag[1,1,2,2,\dots, k,k,\dots]$ (here we use supposition  that all
natural numbers are invertible).

Consider an involution (an element of order~2 in our semigroup)~$J$,
that commutes with some matrix $C\in \Lambda$ and, if it commutes
with some diagonal matrix $C'$, then $C_{\Gamma_n(R)}(C)\subset
C_{\Gamma_n(R)}(C')$. Since $J\in C_{\Gamma_n(R)}(C)$ и $C\in
\Lambda$, then $J$ consists of diagonal blocks $2\times 2$ (and,
possibly, one block  $1\times 1$ in the end). Consider one of these
blocks $J_i=\begin{pmatrix} a&b\\ c& d
\end{pmatrix}$. Since $J_i^2=I$, then $e_1=a^2$ and $e_2=bc$ are orthogonal idempotents with the sum~$1$.
We know that $J_i\cdot diag[\alpha,\beta]\ne diag [\alpha,\beta]
\cdot J_i$ for any invertible $\alpha\ne \beta$ (by the choice
of~$J$). Take $\alpha=1=a^2+bc$ and $\beta = a^2/2+bc$. Then
$\alpha$ and $\beta$ are invertible and $J_i\cdot
diag[\alpha,\beta]=diag [\alpha,\beta]\cdot J_i$. So $\alpha=\beta$,
i.\,e. $a^2=0$. Consequently, $bc=1$. Since $abc=0$, we have $a=0$.
Similarly, $d=0$. Therefore,
$$
J=\begin{pmatrix} 0& b_1&\dots& 0& 0\\
b_1^{-1}& 0& \dots& 0& 0\\
\vdots& \vdots&\ddots& \vdots&\vdots\\
0& 0& \dots& 0& b_k\\
0& 0& \dots& b_k^{-1}& 0 \end{pmatrix} \text{ или }
J=\begin{pmatrix} 0& b_1&\dots& 0& 0&0\\
b_1^{-1}& 0& \dots& 0& 0&0\\
\vdots& \vdots&\ddots& \vdots&\vdots&\vdots\\
0& 0& \dots& 0& b_k&0\\
0& 0& \dots& b_k^{-1}& 0&0\\
0&0& \dots& 0& 0& b \end{pmatrix}.
$$
Conjugating  $J$ by diagonal matrix $diag[b_1,1,b_2,1,\dots,b_k,1]$,
or, respectively, $diag[b_1,b,b_2,b,\dots,b_k,b,1]$, we obtain a
matrix $J'=S_\tau$ or, respectively, $bS_\tau$. So we found a matrix
$M\in \Gamma_n(R)$ such that $M\Phi(S_\tau)M^{-1}=bS_\tau$, $b^2=1$.
\end{proof}

\begin{lemma}\label{n3}
Let $n=3$, $\Phi$ be an automorphism of $G_n(R)$ such that
$\Phi(S_\tau)=bS_\tau$, $b^2=1$, where $\tau$ is a substitution from
the previous lemma \emph{(}in this case it is simply $(1
2)$\emph{)}. Then $\exists M\in \Gamma_{n(R)}$, such that
$\Phi'(S_\rho)=\Phi_M \circ \Phi(S_\rho)=b^{sgn\, \rho} S_\rho,
\forall \rho \in \Sigma_n$.
\end{lemma}

\begin{proof}
From Lemma~2 it follows that
$$
\Phi(diag[\alpha, \alpha, \beta])=diag[\alpha',\alpha',\beta'],
\quad \forall \alpha ,\beta \in R_+^*
$$
since $\Phi(S_{(12)})=bS_{(12)}$ and the first matrix commutes with
$S_{(12)}$. If $\alpha \ne \beta$, then $\alpha' \ne \beta'$, since
$diag[\alpha,\alpha,\beta]$ can not be mapped to a scalar matrix
(the center of the semigroup). Let
$$
\Phi(S_{(23)})=A=\begin{pmatrix}
a_{11}& a_{12} & a_{13}\\
a_{21}& a_{22} & a_{23}\\
a_{31}& a_{32} & a_{33}
\end{pmatrix}
$$
Since $A^2=I$, then, as earlier, we obtain
$$a_{ik}a_{kj}=0,\quad
a_{1i}a_{i1}+a_{2i}a_{i2}+a_{3i}a_{i3}=1\eqno (\ast)$$ Note that $A$
does not commute with $diag[\alpha',\alpha',\beta']$. Therefore we
obtain that at least one of the following conditions
$$
\alpha a_{13} \ne \beta a_{13},\quad \alpha a_{23} \ne \beta
a_{23},\quad \alpha a_{31} \ne \beta a_{31}, \quad\alpha a_{32} \ne
\beta a_{32}
$$
is satisfied. Since $A^2=I$, then, as in the previous lemma,
$\xi=a_{12}a_{21}$ is an idempotent. Suppose that it is nonzero.
Then let us take ${\alpha'=1+\xi}$, $ {\beta'=1}$ ($\alpha'$ is
invertible, and the inverse element is $1-\xi/2$). Then no one of
these above conditions is satisfied. We come to  contradiction, so
$a_{12}a_{21}=0$. Since $A^2=I$ we have $a_{11}^2+a_{13}a_{31}=1$.
Multiplying this equality to $a_{12}$, we obtain $a_{12}=0$.
Similarly, $a_{21}=0$. Now we use the fact that
$S_{(12)}S_{(23)}=S_{(123)}$, i.e. it has order $3$, and
consequently $(S_{(12)}A)^3=I$. Thus, using $(\ast)$, we obtain the
conditions
$$
a_{11}a_{23}a_{32}+a_{22}a_{13}a_{31}=b,\quad a_{11}^2a_{22}=0,\quad
a_{11}a_{22}^2=0,\eqno(\ast\ast)
$$
From  $(\ast)$ it follows $a_{11}^3=a_{11}, a_{22}^3=a_{22},
a_{11}a_{23}a_{32}=a_{11}, a_{22}a_{13}a_{31}=a_{22}$, hence
$a_{11}+a_{22}=b$. From the other side,
$a_{11}a_{23}a_{32}+a_{22}a_{13}a_{31}+a_{33}^3=b$, and therefore
$a_{33}^3=0=a_{33}$. Consequently we can rewrite $(\ast)$ in the
form
$$
a_{13}a_{31}+a_{23}a_{32}=1, \quad a_{11}^2+a_{31}a_{13}=1,\quad
a_{22}^2+a_{32}a_{23}=1
$$
Denote $e_1=a_{23}a_{32},\quad e_2=a_{13}a_{31}$. These are
orthogonal idempotents with the sum $1$. Consequently we know that
${a_{11}e_1+a_{22}e_2=b}$ (it is one of equalities $(\ast\ast)$). So
$a_{11}e_1=be_1$, but since $a_{11}$ is orthogonal to $e_2$, then it
belongs to the ring $e_1R$, and so $a_{11}=ba_{11}e_1=be_1$.
Similarly we obtain $a_{22}=be_2$, and the matrix $A$ has the form
$$
A=\begin{pmatrix}
be_{1}& 0 & a_{13}\\
0& be_{2} & a_{23}\\
a_{31}& a_{32} & 0
\end{pmatrix}
$$
Take now the matrix
$$ C=\begin{pmatrix}
e_{1}& e_{2} & 0\\
e_{2}& e_{1} & 0\\
0& 0 &1
\end{pmatrix}.
$$
It is invertible (inverse to itself) and commutes with $S_{(12)}$,
therefore under the automorphism $\Phi_C$ the matrix $S_{(12)}$ is
mapped to $bS_{(12)}$. The matrix $A$ under this automorphism is
mapped to the matrix
$$
A'=\begin{pmatrix}
b& 0 & 0\\
0& 0 & b(a_{13}+a_{23})\\
0& b(a_{31}+a_{32}) & 0
\end{pmatrix}
$$
Note that $(a_{13}+a_{23})(a_{31}+a_{32})=e_1+e_2=1$. Now let us
take the matrix $C'=diag[1,1,a_{13}+a_{23}]$ (it also commutes with
$S_{(12)}$) and conjugate  $A'$ by this matrix. We obtain
$bS_{(23)}$. Since the  symmetric group is generated by
substitutions $(12)$ and $(23)$, we get to the obtained automorphism
$\Phi'=\Phi_M\circ\Phi$. The lemma is proved.
\end{proof}

\begin{lemma}
Let  $n=4$, $\Phi$ be an automorphism of the semigroup of
nonnegative matrices such that $\Phi(S_\tau)=S_\tau$, where $\tau$
is a substitution from Lemma~\emph{3} \emph{(}in our case it is  $(1
2)(34))$. Let $\rho= (13)(24)$. Then $\exists M\in \Gamma_{n(R)}$,
such that $\Phi_M \circ \Phi(S_\rho)=S_\rho$ and $\Phi_M\circ
\Phi(S_\tau)=S_\tau$.
\end{lemma}
\begin{proof}
Let $X=\Phi(S_{\rho})$. Since $S_\tau$ commutes with $S_\rho$, then
$$
X=\begin{pmatrix}
a_1& a_2& b_1 & b_2\\
a_2& a_1 & b_2 &b_1\\
c_1& c_2 & d_1 &d_2\\
c_2 &c_1 & d_2 &d_1
\end{pmatrix}.
$$
We will consider $X$ as a matrix $2\times 2$ over matrices $2\times
2$:
$$X=\begin{pmatrix} A& B\\
C& D
\end{pmatrix},\quad A,B,C,D\in M_2(R).
$$
Since $X^2=I$, we have $A^2+BC=I_2$, and since $A=\begin{pmatrix}
a_1& a_2\\ a_2& a_1\end{pmatrix}$, we have $A^2=\begin{pmatrix}
a_1^2+a_2^2& 2a_1a_2\\ 2a_1a_2& a_1^2+a_2^2\end{pmatrix}$. Therefore
the matrices $A^2$ and (similarly) $BC$ have equal elements on the
diagonal. Since the sum of matrices $A^2$ and $BC$ has zeros outside
of the main diagonal, then the matrices $A^2$ and $BC$ are diagonal
(and even are scalar ). Consequently, they are central idempotents
of the matrix ring $2\times 2$. Note later that the matrix
$diag[d_1,d_1,d_2,d_2]$, where $d_1\ne d_2,\,\,\, d_1,d_2\in R^*_+$,
is mapped under this automorphism to the matrix of the same form,
i.e. $diag[d'_1,d'_1,d'_2,d'_2],\,\,\,d'_1\ne d'_2$, because a given
matrix is diagonal, commutes with $S_\tau$ and is not scalar. Let us
consider the matrix $\begin{pmatrix} A^2/2+BC& 0\\
0& I_2
\end{pmatrix}$. It has the form described above and commutes with $X$ (to check it we use the fact that
$B$ and $C$ commute, and also use the equalities $B^2C=B$ and
$BC^2=C$, that are obtained from $A^2+BC=I_2$ by multiplying to $B$
and $C$). Consequently, $A^2/2+BC=I_2=A^2+BC,\,\, A^2=0,\,\,BC=I_2$,
and since $ABC=0$, then $A=0$, similarly $D=0$. Now it is clear that
the obtained matrix is
$M=\begin{pmatrix} C& 0\\
0& I_2
\end{pmatrix}$ (note, that it commutes with $S_\tau$).
Therefore lemma is proved.
\end{proof}

\begin{lemma}\label{n4}
Let $n=4$, $\Phi$ be an automorphism constructed in the previous
lemma.
 Then $\exists M\in \Gamma_n(R)$ and an involution $a\in R^*_+$, such that
$\Phi_M \circ \Phi(S_\sigma)=a^{sgn\, \sigma}S_\sigma$ for any
substitution~$\sigma$.
\end{lemma}
\begin{proof}
Let $Y=\Phi(S_{(4321)})$, where $\Phi$ is an automorphism from the
previous lemma. Since $S_{(4321)}$ commutes with $S_{(13)(24)}$, we
have that $Y$ has the form
$$
Y=\begin{pmatrix}
y_1& z_1& y_2 &z_2\\
x_1& w_1 & x_2 &w_2\\
y_2& z_2 & y_1 &z_1\\
x_2 &w_2 & x_1 &w_1
\end{pmatrix}.
$$
Now use the identity
$$
 S_{(4321)}S_{(12)(34)}=S_{(14)(23)}S_{(4321)},
$$
that implies $x_1=z_2$, $z_1=x_2$, $y_1=w_2$, $w_1=y_2$. Therefore,
$$
Y=\begin{pmatrix}
y_1& z_1& y_2 &z_2\\
z_2& y_2 & z_1 &y_1\\
y_2& z_2 & y_1 &z_1\\
z_1 &y_1 & z_2 &y_2
\end{pmatrix}.
$$
Finally, use $Y^2=S_{(13)(24)}$. We get the conditions
$y_1^2+2z_1z_2+y_2^2=0,\,\,y_1z_2+y_1z_1+y_2z_1+y_2z_2=0,\,\,
z_1^2+2y_1y_2+z_2^2=1$. Multiplying the last one to $y_1$, we obtain
$z_1^2y_1+2y_1^2y_2+z_2^2y_1=y_1$, but from the first equalities it
follows $y_1^2=0,\,\,y_1z_1=0,\,\,y_1z_2=0$ (see Lemma~1).
Therefore, $y_1=0$, similarly $y_2=0$. Consequently,
$$
Y=\begin{pmatrix}
0& z_1& 0 &z_2\\
z_2& 0 & z_1 &0\\
0& z_2 & 0 &z_1\\
z_1 &0 & z_2 &0
\end{pmatrix}=\begin{pmatrix}
0& z_1& 0 &0\\
0& 0 & z_1 &0\\
0& 0 & 0 &z_1\\
z_1 &0 & 0 &0
\end{pmatrix}
+\begin{pmatrix}
0& 0& 0 &z_2\\
z_2& 0 & 0 &0\\
0& z_2 & 0 &0\\
0 &0 & z_2 &0
\end{pmatrix}.
$$
We have $z_1^2+z_2^2=1$ and $z_1z_2=0$. Set
$M_1=z_2^2S_{(14)(23)}+z_1^2I_4$ (it is invertible, inverse to
itself). Let $\Phi_1=\Phi_{M_1}\circ \Phi$, then (since $M_1$
commutes with $S_\tau$ and $S_\rho$)
$$\Phi_1(S_{(4321)})=(z_1+z_2)S_{(4321)},\,\,\Phi_1(S_\tau)=S_\tau,\,\,\Phi_1(S_\rho)=S_\rho.$$

Now we will consider the matrix $A=\Phi_1(S_{12})$. Then from the
conditions of commutativity $S_{12}$ with $S_{(12)(34)}$ and with
$diag[\alpha,\alpha,\beta,\beta]$ (the last one under our
automorphism is mapped to some similar matrix if $\alpha, \beta \in
R^*_+$: see the previous lemma) we obtain
$$
A=\begin{pmatrix}
a_1& a_2& 0 &0\\
a_2& a_1 & 0 &0\\
0& 0 & b_1 &b_2\\
0 &0 & b_2 &b_1
\end{pmatrix}.
$$
The identity
$$
S_{(13)(24)}S_{(12)}S_{(13)(24)}S_{(12)}=S_{(12)(34)}
$$
implies the conditions for matrix elements
$$
a_1b_1+a_2b_2=0, \quad a_1b_2+a_2b_1=1.
$$
From the first equality we obtain $a_1b_1=a_2b_2=0$. Let us multiply
the second equality to $a_1$, then $a_1^2b_2+a_1a_2b_1=a_1$. Since
$X^2=I_4$ ($a_1a_2=0 ,\,\,a_1^2=1-a_2^2$), we have
$a_1=b_2(1-a_2^2)=b_2$. Similarly, $a_2=b_1$. Therefore,
$$
A=\begin{pmatrix}
a_1& a_2& 0 &0\\
a_2& a_1 & 0 &0\\
0& 0 & a_2 &a_1\\
0 &0 & a_1 &a_2
\end{pmatrix}=\begin{pmatrix}
0& a_2& 0 &0\\
a_2& 0 & 0 &0\\
0& 0 & a_2 &0\\
0 &0 & 0 &a_2
\end{pmatrix}
+\begin{pmatrix}
a_1& 0& 0 &0\\
0& a_1 & 0 &0\\
0& 0 & 0 &a_1\\
0 &0 & a_1 &0
\end{pmatrix}.
$$
Set $M'=a_1^2S_{(13)(24)}+a_2^2I_4$ (it is invertible, inverse to
itself). Let $\Phi_2=\Phi_{M'}\circ \Phi_1$, then (since $M'$
commute with $S_{(4321)}$,$S_\tau$ and $S_\rho$)
$$\Phi_2(S_{(12)})=(a_1+a_2)S_{(12)},\,\,
\Phi_2(S_{(4321)})=(z_1+z_2)S_{(4321)},\,\,
\Phi_2(S_\tau)=S_\tau,\,\,\Phi_2(S_\rho)=S_\rho.$$ Now it remains to
prove that involutions $a_1+a_2$ and $z_1+z_2$ coincide. Let
$X=\Phi_2(S_{(432)})=(a_1+a_2)(z_1+z_2)S_{(432)}$. Since
$(432)=(4321)(12)$, we have $X^3=I_4$, consequently
$(a_1+a_2)^3(z_1+z_2)^3=1$. Therefore $(a_1+a_2)(z_1+z_2)=1$, and so
$z_1+z_2=(a_1+a_2)^{-1}=a_1+a_2$.
\end{proof}

\begin{definition}\label{bloki}
Let a number $n$ be decomposed into the sum of powers of~$2$:
$$
n=2^{k_1}+2^{k_2}+\dots+2^{k_l},\quad k_1\ge k_2\ge \dots \ge k_l\ge
0.
$$
A diagonal block  $2^{k_i}\times 2^{k_i}$, $i=1,\dots,l$, of the
size $n\times n$, corresponding to basis vectors with numbers
$2^{k_1}+2^{k_2}+\dots+2^{k_{i-1}}+1,\dots,
2^{k_1}+2^{k_2}+\dots+2^{k_i}$, is denoted by $\mathbf{D_i}$.
\end{definition}

\begin{definition}\label{sigmy}
By $\sigma_i^{(j)}$, $j=1,\dots,l$, $i=1,\dots,k_i$, we denote the
substitution that acts identically on all blocks  $\mathbf{D_m}$,
except the block $\mathbf{D_j}$, and in the block $\mathbf{D_j}$ it
is the product of $2^{k_j-1}$ transpositions, every of them is
$(p,p+2^{i-1})$.

By $\sigma_i$ we denote the substitution $\sigma_{i_1}^{(1)}\cdot
\dots \cdot \sigma_{i_l}^{(l)}$, where $i_q=\min(i,k_q)$.
\end{definition}

For example, for $n=7$
$$
\sigma_1=(1,2)(3,4)(5,6),\ \sigma_2=(1,3)(2,4)(5,6),
$$
for $n=10$ \begin{align*}
 \sigma_1&=(1,2)(3,4)(5,6)(7,8)(9,10),\\
\sigma_2&=(1,3)(2,4)(5,7)(6,8)(9,10),\\
\sigma_3&=(1,5)(2,6)(3,7)(4,8)(9,10).
\end{align*}

\begin{lemma}\label{sigmai}
 Let $\Phi$ be an arbitrary automorphism of $G_n(R)$.
 Then $\exists M\in
\Gamma_n(R)$, such that $\Phi_M \circ
\Phi(S_{\sigma_i})=a_iS_{\sigma_i}$, $a_i^2=1$, $\forall
i=1,2,\dots$.
\end{lemma}
\begin{proof}
By Lemma~\ref{l1} we can choose a matrix $M_1$ such that
$\Phi_1(S_{\sigma_1})=\Phi_{M_1} \circ
\Phi(S_{\sigma_1})=S_{\sigma_1}$.

Now consider the matrix $A_2=\Phi_1(S_{\sigma_2})$. Let $n$ be odd.
Since it commutes with $S_{\sigma_1}$, we have that
$$
A_2=\begin{pmatrix} a_{11} & a_{12} & a_{13}& a_{14}& \dots &
a_{1,n-2} & a_{1,n-1} & a_{1,n}\\
a_{12}& a_{11} & a_{14} & a_{13} &\dots & a_{1,n-1}&
a_{1,n-2} &a_{1,n}\\
a_{31}& a_{32}& a_{33}& a_{34} &\dots & a_{3,n-2}&
a_{3,n-1} & a_{3,n}\\
a_{32}& a_{31} & a_{34}& a_{33} & \dots& a_{3,n-1}&
a_{3,n-2} &a_{3,n}\\
\vdots& \vdots& \vdots& \vdots & \ddots& \vdots& \vdots\\
a_{n-2,1}& a_{n-2,2}& a_{n-2,3} & a_{n-2,4}& \dots & a_{n-2,n-2}&
a_{n-2,n-1} &a_{n-2,n}\\
a_{n-2,2}& a_{n-2,1}& a_{n-2,4}& a_{n-2,3}&\dots& a_{n-2,n-1}&
a_{n-2,n-2} &a_{n-2,n}\\
a_{n,1} &a_{n,1} & a_{n,3} & a_{n,3} &\dots& a_{n,n-2} & a_{n,n-2} &
a_{n,n}\end{pmatrix}.
$$

Then from $A_2^2=I$ it follows (if we consider the  right column
multiplying to the second row, the third column multiplying to the
fourth row, etc.) $a_{n,i}a_{i,n}=0$ for all $i=1,\dots,n-1$. Now
consider the last column, multiplying by the last row. From the
obtained equalities it follows $a_{n,n}^2=1$, i.\,e. $a_{n,n}$ is an
invertible element of~$R$. Therefore $a_{i,n}=a_{n,i}=0$ for all
$i=1,\dots, n-1$. So we can bound a given matrix for the size
$(n-1)\times (n-1)$. Consequently, we can suppose yet that the
semigroup dimension is even. In this case the matrix has the form
$$
A_2=\begin{pmatrix} a_{11} & a_{12} & a_{13}& a_{14}& \dots &
a_{1,n-1} & a_{1,n} \\
a_{12}& a_{11} & a_{14} & a_{13} &\dots & a_{1,n}&
a_{1,n-1}\\
a_{31}& a_{32}& a_{33}& a_{34} &\dots & a_{3,n-1}&
a_{3,n}\\
a_{32}& a_{31} & a_{34}& a_{33} & \dots& a_{3,n}&
a_{3,n-1}\\
\vdots& \vdots& \vdots& \vdots & \ddots& \vdots& \vdots\\
a_{n-1,1}& a_{n-1,2}& a_{n-1,3} & a_{n-1,4}& \dots & a_{n-1,n-1}&
a_{n-1,n}\\
a_{n-1,2}& a_{n-1,1}& a_{n-1,4}& a_{n-1,3}&\dots& a_{n-1,n}&
a_{n-1,n-1} \end{pmatrix}. $$

We will consider $A_2$ as a matrix $n/2\times n/2$ over matrices
$2\times 2$:
$$A_2=\begin{pmatrix} A_2^{(1,1)}& \dots& A_2^{1,n/2}\\
\vdots& \ddots & \vdots\\
A_2^{(n/2,1)}& \dots& A_2^{(n/2,n/2)}
\end{pmatrix},\quad A_2^{(1,1)},\dots,A_2^{(n/2,n/2)}\in M_2(R).
$$
Since $A_2^2=1$, we have
${A_2^{(1,1)}}^2+A_2^{(1,2)}A_2^{(2,1)}+\dots+A_2^{(1,n/2)}A_2^{(n/2,1)}=I_2$,
and since $A_2^{(i,j)}=\begin{pmatrix} a& b\\ b&
a\end{pmatrix}$, we have $A_2^{(1,i)} A_2^{(i,1)}=\begin{pmatrix} a_1a_2+b_1b_2& a_1b_2+a_2b_1\\
a_1b_2+a_2b_1& a_1a_2+b_1b_2\end{pmatrix}$. Consequently the matrix
$A_2^{(1,i)}A_2^{(i,1)}$ has equal elements on the diagonal. Since
the sum of all matrices $A_2^{(1,i)}A_2^{(i,1)}$ has zeros outside
the main diagonal, then all matrices $E_i=A_2^{(1,i)}A_2^{(i,1)}$
are diagonal, therefore they are scalar. Hence they are central
idempotents of the matrix ring $2\times 2$. Therefore $\{E_i\}$ is a
system of mutually orthogonal idempotents with the sum~$1$, i.\,e.
$M_2(R)=E_1M_2(R)\oplus E_2M_2(R)\oplus \dots\oplus E_{n/2}M_2(R)$.
Let us represent $A_2$ as $E_1A_2+\dots +E_{n/2}A_2$.

Conjugate the matrix $A_2$ by the matrix
$B=E_1B+\dots+E_nB=B_1+\dots+B_n$, where $B_1=E_1\cdot I$,
$B_i=E_i\cdot S_{(3,2i-1)(4,2i)}$. Then
$A_2'=B^{-1}A_2B=E_1A_2'+\dots E_nA_2'={A_2^{(1)}}'+\dots +
{A_2^{(n)}}'$, where
\begin{align*}
{A_2^{(1)}}'&=\begin{pmatrix} A_2^{(1,1)}& 0& \dots& 0\\
0& \alpha_{11}^{(1)}& \dots& \alpha_{1,n-1}^{(1)}\\
0& \vdots& \ddots&\vdots\\
0& \alpha_{n-1,1}^{(1)}& \dots& \alpha_{n-1,n-1}^{(1)}
\end{pmatrix},\\
{A_2^{(i)}}'&=\begin{pmatrix} 0& A_{1i}& 0&\dots& 0\\
A_{i1}& 0& 0& \dots& 0\\
0& 0& \alpha_{11}^{(i)}& \dots & \alpha_{1,n-2}^{(i)}\\
\vdots& \vdots&\vdots&\ddots& \vdots\\
0& 0& \alpha_{n-2,1}^{(i)}& \dots& \alpha_{n-2,n-2}^{(i)}
\end{pmatrix}, \quad i>1.
\end{align*}
Let us denote the matrix $\begin{pmatrix} \alpha_{11}&\dots &
\alpha_{1,n-2}\\
\vdots& \ddots& \vdots\\
\alpha_{n-2,1}& \dots& \alpha_{n-2,n-2} \end{pmatrix}$ (and in the
case $i=1$ the matrix $\begin{pmatrix} \alpha_{11}&\dots &
\alpha_{1,n-1}\\
\vdots& \ddots& \vdots\\
\alpha_{n-1,1}& \dots& \alpha_{n-1,n-1} \end{pmatrix}$)
by~$\Lambda_i$. Note that $\Lambda_i$ is a matrix over the ring
$E_iM_2(R)$ of the size $<n/2$ such that $\Lambda_i^2=I$. Therefore,
we can repeat previous arguments for the matrix $\Lambda_i$ and the
corresponding system of orthogonal idempotents of $E_iM_2(R)$.
Finally we obtain the matrix~$\widetilde A_2$, conjugate to the
initial matrix~$A_2$ and consisting of diagonal blocks  $4\times 4$
and $2\times 2$. So every element of the order in $\Gamma_n(R)$,
commuting with $S_{\sigma_1}$, in some basis (where $S_{\sigma_1}$
is not changed) consists of diagonal blocks  $4\times 4$ and
$2\times 2$ (it is clear that for matrices of odd sizes there can be
one block $1\times 1$).

Consider now the set
\begin{multline*}
 {\mathcal F}_1=\{ D\in D_n(R)\mid DS_{\sigma_1}=S_{\sigma_1}D \land \forall J\in \Gamma_n(R)
 (J^2=I\land JS_{\sigma_1} =\\
 =S_{\sigma_1} J\Rightarrow \exists C\in \Gamma_n(R) (D(CJC^{-1})=
 (CJC^{-1}D)\}.
 \end{multline*}
This set consists of diagonal matrices with pairs of equal elements
on the diagonal that commute in some basis with all matrices of the
order~$2$, that can be represented as blocks $4\times 4$ and
$2\times 2$. Therefore every element $D$ of this set in some basis
has the form
 $$
 diag[\alpha_1I_2,\alpha_1I_2,\alpha_2I_2,\alpha_2I_2,\dots ,
 \alpha_kI_2,\alpha_kI_2,I_2],
 $$
  if $n/2=2k+1$, or
  $$
  diag
 [\alpha_1I_2,\alpha_1I_2,\dots ,\alpha_kI_2,\alpha_kI_2],
 $$
  if $n/2=2k$, (if the dimension of the semigroup is odd, then we add one more diagonal element in the end) ($*$), and the corresponding basis change
  does not move $S_{\sigma_1}$, $D$
  commutes with $S_{\sigma_1}$ and with $S_{\sigma_2}$ in some basis.

 Consider the set
 $$
 \Lambda_1=\{ D\in {\mathcal F}_1\mid \forall D'\in {\mathcal F}_1
 C_{\Gamma_n(R)} (D) \not\supset C_{\Gamma_n(R)}(D')\}.
 $$

Every matrix from ${\mathcal F}_1$ commutes (in a basis where it has
the form ($*$)) with all matrices that in this basis are divided
into diagonal blocks  $4\times 4$ (and, possibly, $2\times 2$ in the
end and also, possibly, $1\times 1$ in the end). So $\Lambda_1 $
contains  matrices with centralizers only  of these matrices. From
the other side,  matrices with these properties exist in every
basis, for example these are matrices $diag[1,1,1,1,2,2,2,2\dots,
k,k,k,k,\dots]$.

Consider the involution~$K$ that is the image of $S_{\sigma_2}$.
This matrix commutes with $S_{\sigma_1}$ and with some matrix $C\in
\Lambda_1$. Moreover, if it commutes with some diagonal matrix $C'$,
then $C_{\Gamma_n(R)}(C)\subset C_{\Gamma_n(R)}(C')$. Since $K\in
C_{\Gamma_n(R)}(C)$ and $C\in \Lambda_1$, then $J$ consists of
diagonal blocks $4\times 4$ (and, possibly, one block $2\times 2$,
and also possibly one block $1\times 1$).

Consider one of these blocks ($4\times 4$)
$$
\widetilde K_i=\begin{pmatrix} A&B\\
C& D
\end{pmatrix},\quad A,B,C,D\in M_2(R).
$$

 Since $\widetilde K_i$ is an involution commuting with  $S_{\sigma_1}$ bounded on a given part of basis, then $E_1=A^2$ and $E_2=BC$
 are central orthogonal idempotents of $M_2(R)$ with the sum~$I_2$. We know that $\widetilde K_i\cdot
diag[\alpha,\alpha,\beta,\beta]\ne diag [\alpha,\alpha,\beta,\beta]
\cdot \widetilde K_i$ for any invertible $\alpha\ne \beta$ (by the
choice of~$K$). Take $\alpha I_2=A^2+BC$ and $\beta I_2 = A^2/2+BC$.
Then $\alpha$ and $\beta$ are invertible and $\widetilde K_i\cdot
diag[\alpha,\alpha,\beta, \beta]=diag [\alpha,\alpha, \beta,
\beta]\cdot \widetilde K_i$. Therefore, $\alpha=\beta$, i.\,e.
$A^2=0$. Consequently, $BC=1$. Since $ABC=0$, we have $A=0$.
Similarly $D=0$. Therefore $K$ (depending of dimension) can have one
of four forms:
\begin{gather*}
\begin{pmatrix} 0& B_1&\dots& 0& 0\\
B_1^{-1}& 0& \dots& 0& 0\\
\vdots& \vdots&\ddots& \vdots&\vdots\\
0& 0& \dots& 0& B_k\\
0& 0& \dots& B_k^{-1}& 0 \end{pmatrix}, \qquad
\begin{pmatrix} 0& B_1&\dots& 0& 0&0\\
B_1^{-1}& 0& \dots& 0& 0&0\\
\vdots& \vdots&\ddots& \vdots&\vdots&0\\
0& 0& \dots& 0& B_k&0\\
0& 0& \dots& B_k^{-1}& 0&0\\
0& 0& \dots& 0& 0& b_{k+1} \end{pmatrix}\\
\begin{pmatrix} 0& B_1&\dots& 0& 0&0&0\\
B_1^{-1}& 0& \dots& 0& 0&0&0\\
\vdots& \vdots&\ddots& \vdots&\vdots&0&0\\
0& 0& \dots& 0& B_k&0&0\\
0& 0& \dots& B_k^{-1}& 0&0&0\\
0& 0& \dots& 0& 0& 0& b_{k+1}\\
0& 0& \dots& 0& 0& b_{k+1}^{-1}&0
 \end{pmatrix},\\
\begin{pmatrix} 0& B_1&\dots& 0& 0&0&0&0\\
B_1^{-1}& 0& \dots& 0& 0&0&0&0\\
\vdots& \vdots&\ddots& \vdots&\vdots&0&0&0\\
0& 0& \dots& 0& B_k&0&0&0\\
0& 0& \dots& B_k^{-1}& 0 &0&0&0\\
0&0&\dots& 0&0&0&b_{k+1}&0\\
0&0&\dots&0&0&b_{k+1}^{-1}&0&0\\
0&0& \dots&0&0&0&0&b_{k+2}
\end{pmatrix}.
\end{gather*}
After conjugating  $K$ by the block-diagonal matrix that has the
form
$$
diag[B_1,I_2,B_2,I_2,\dots,B_k,I_2]
$$
 in the first case,
$$
diag[B_1,b_{k+1}I_2,B_2,b_{k+1}I_2,\dots, B_k,b_{k+1}I_2,1]
$$
 in the second case,
 $$
 diag[B_1,I_2,B_2,I_2,\dots,
B_k,I_2,b_{k+1},1]
$$
 in the third case,
$$
diag[B_1,b_{k+2}I_2,B_2,b_{k+2}I_2,\dots,B_k,b_{k+2}I_2,
b_{k+1},b_{k+2},1]
$$
 in the fourth case, we obtain the matrix
$K'=S_{\sigma_2}$, or $b_{k+1}S_{\sigma_2 }$, $b_{k+2}S_{\sigma_2}$,
and the matrix $S_{\sigma_1}$ is not changed. Therefore, we found a
matrix $M_2\in \Gamma_n(R)$ such that
$M\Phi(S_{\sigma_i})M_2^{-1}=b_iS_{\sigma_i}$, $i=1,2$. Continuing
this procedure we come to the obtained basis change.
\end{proof}

\begin{definition}\label{scalblocks}
\emph{We call a diagonal matrix $D\in \Gamma_n(R)$ \emph{a
block-scalar involution}, if $D^2=I$ and  $D$ is scalar in every
block~$\mathbf{D_i}$, $i=1,\dots,l$. The set (group) of all
block-scalar involutions is denoted by~$\mathcal Q$.}
\end{definition}

\begin{lemma}\label{sigmaij}
Let an automorphism  $\Phi$ of $G_n(R)$ be such that
$\Phi(S_{\sigma_i})=a_iS_{\sigma_i}$, $a_i^2=1$, $i=1,\dots, k_1$.
Then every matrix $S_{\sigma_i^{(j)}}$  under the action of~$\Phi$
is mapped to $D\cdot S_{\sigma_i^{(j)}}$, where $D\in{\mathcal Q}$.
\end{lemma}

\begin{proof}
We will prove this statement by induction by sizes of blocks
$\mathbf{D_j}$.

{\bf Induction basis.} For a block  $\mathbf{D_l}$ of the size
$1\times 1$ we do not need to prove anything, therefore we start
with the block $2\times 2$ (if it exists).

Suppose that the block $2\times 2$ has the place $p,p+1$. Consider
the matrix $S_{\sigma_1^{(l)}}=S_{(p,p+1)}$. It satisfies the
following properties:

1) $S_{(p,p+1)}$ commutes with all $\sigma_i$, $i=1,\dots,k_1$;

2) $S_{(p,p+1)}^2=I$;

3) if some diagonal matrix~$D$ commutes with some $S_{\sigma_i}$,
then $D$ commutes also with $S_{(p,p+1)}$;

4) if a diagonal matrix commutes with $S_{(p,p+1)}$, then it also
commutes with any other matrix satisfying the properties 1--3.

Since any matrix that is  scalar in every block $\mathbf{D_i}$,
commutes with $S_{\sigma_i}$, $i=1,\dots,k_1$, then the image
$A_{(p,p+1)}=\Phi(S_{(p,p+1)}$ has to commute, for example, with the
matrix $diag[I_{2^{k_1}}, 2\cdot I_{2^{k_2}},\dots, l\cdot
I_{2^{k_l}}]$, therefore the matrix $A_{(p,p+1)}$ can also be
divided into blocks $\mathbf{D_1},\dots, \mathbf{D_l}$. Consider the
matrix $A_{(p,p+1)}$ and its block $\mathbf{D_i}$ of the size
greater $2\times 2$. In this block there are exactly  $k_i$
different $S_{\sigma_m^{(i)}}$. Consider a set of diagonal matrices
$H_1,\dots, H_{k_i}$, identical in all blocks, except
$\mathbf{D_i}$, and in the block $\mathbf{D_i}$
\begin{align*} H_1&=diag[1,1,2,2,3,3,\dots,
2^{k_i-1},2^{k_i-1}],\\
H_2&=diag[1,2,1,2,3,4,3,4,\dots, 2^{k_i-1}-1, 2^{k_i-1},2^{k_i-1}-1,
2^{k_i-1}],\\
H_3&=diag[1,2,3,4,1,2,3,4,\dots,],\dots, \\
H_{k_i}&=diag[1,2,\dots, 2^{k_i-1},1,2,3,\dots, 2^{k_i-1}].
\end{align*}

 Every $H_j$ commutes with a corresponding matrix $S_{\sigma_j}$, therefore (the condition~3) matrix $A_{(p,p+1)}$
commutes with all $H_j$, $j=1,\dots, k_i$. Thus $A_{(p,p+1)}$ in the
block $\mathbf{D_i}$ is diagonal. From the condition~1 it follows
that $A_{(p,p+1)}$ in the block $\mathbf{D_i}$ is scalar (since it
has the order~2, then the corresponding scalar number is an
involution). Now let us look at the last block (of the size $2\times
2$). Let
in this block $A_{(p,p+1)}$ have the form $\begin{pmatrix} a&b\\
c& d\end{pmatrix}$. Since $A_{(p,p+1)}$ commutes with
$S_{\sigma_1}$, then $a=d, b=c$. Since $A_{(p,p+1)}$ has the
order~$2$, then $a^2+b^2=1$, $ab=0$.

Use now the condition~4.

Consider the diagonal matrix $H$, identical in all blocks
$\mathbf{D_i}$, except the block under consideration, and in the
block under consideration $2\times 2$ having the form $diag[1,
a^2/2+b^2]$. Then
$$
\begin{pmatrix}
1\cdot a& 1\cdot b\\
(a^2/2+b^2)\cdot b& (a^2/2+b^2)\cdot a
\end{pmatrix}=
\begin{pmatrix}
1\cdot a& (a^2/2+b^2)\cdot b\\
1\cdot b& (a^2/2+b^2)\cdot a
\end{pmatrix},
$$
since $ab=0$, i.\,e. $H$ and $A_{(p,p+1)}$ commute. Therefore, $H$
commutes with any other matrix, satisfying the conditions 1--3, and
also with $S_{(p,p+1)}$. It means that $a^2/2+b^2=1$, i.\,e.
$a^2=0$. Consequently, $a=0$ (since from $a^2+b^2=1$ and $ab=0$ it
follows $a^3=a$). Thus $A_{(p,p+1)}=DS_{(p,p+1)}$, where $D\in
{\mathcal Q}$.

{\bf Induction step.} Now we can suppose that the matrix of every
substitution $\sigma_m^{(j)}$, $j> i$, $m=1,\dots, k_{i+1}$ is
mapped  under our automorphism into itself multiplied by some matrix
from~$\mathcal Q$. Under this supposition let us consider now the
matrices of substitutions $\sigma_m^{(i)}$, $m=1,\dots, k_i$. Images
of all such matrices commute with all diagonal matrices commuting
with all $S_{\sigma_1},\dots, S_{\sigma_{k_1}}$, and therefore with
the matrix $diag[I_{2^{k_1}}, 2\cdot I_{2^{k_2}},\dots, l\cdot
I_{2^{k_l}}]$, consequently all these images are also divided into
the blocks $\mathbf{D_1},\dots, \mathbf{D_l}$.

 Let $A_{\sigma_1^{(i)}}=\Phi(S_{\sigma_1^{(i)}})$.
Since $A_{\sigma_1^{(i)}}$ commutes with all diagonal matrices
commuting with $S_{\sigma_1}$, we have that $A_\sigma$ is divided
into diagonal blocks $2\times 2$. Now we can say that  $A_\sigma$
commutes with any diagonal matrix commuting with
$\sigma_1^{(1)}\cdot \dots \cdot \sigma_1^{(i)}=\sigma_1\cdot
\sigma_1^{(i+1)}\dots \sigma_1^{(l)}$. Such diagonal matrices can
have any arbitrary elements on the diagonal, starting from the block
$\mathbf{D_{i+1}}$, therefore the matrix $A_{\sigma_1^{(i)}}$ on the
blocks $\mathbf{D_{i+1}},\dots, \mathbf{D_l}$ is diagonal. According
to the fact that it commutes with all $S_{\sigma_m^{(j)}}$, $j> i$,
we obtain that in every block $\mathbf{D_j}$, $j> i$, the matrix
$A_{\sigma_1^{(i)}}$ is scalar.

The matrix $A_{\sigma_1^{(i)}}$ commutes with all
$S_{\sigma_1},\dots, S_{\sigma_{k_1}}$, therefore inside  every
$\mathbf{D_j}$ all blocks $2\times 2$ are equal.

 Now consider
some block $\mathbf{D_j}$, $j< i$. The matrices $S_{\sigma_i},
S_{\sigma_{i+1}},\dots, S_{\sigma_{k_1}}$ coincide on the block
$\mathbf{D_i}$, hence all their pairwise products are identical on
$\mathbf{D_i}$. So the matrix $A_{\sigma_1^{(i)}}$ commutes with all
diagonal matrices, commuting with one of
$S_{\sigma_1}S_{\sigma_i}S_{\sigma_{i+1}}$,
$S_{\sigma_1}S_{\sigma_i}S_{\sigma_{i+2}}$,\dots,
$S_{\sigma_1}S_{\sigma_i}S_{\sigma_{k_1}}$. For example, we can take
the matrix
$$diag[1,2,1,2,\dots, 2,1,2,1]$$
 on $\mathbf{D_j}$, it  commutes with $A_{\sigma_1^{(i)}}$, therefore on the block $\mathbf{D_j}$ our $A_{\sigma_1^{(i)}}$
is diagonal, consequently it is scalar.

Now we only need to consider the matrix $A_{\sigma_1^{(i)}}$ on the
block $\mathbf{D_i}$. So $A_{\sigma_1^{(i)}}$ commutes with all
$S_{\sigma_i}$, $i=1,\dots, k_1$. We know that $A_{\sigma_1^{(i)}}$
on $\mathbf{D_i}$ consists of the same diagonal blocks $2\times 2$ of the form $\begin{pmatrix} a& b\\
b& a\end{pmatrix}$, and moreover (since $A_{\sigma_1^{(i)}}$ has
order~2), $a^2+b^2=1$, $ab=0$. Now we can note that the matrix
$S_{\sigma_1^{(i)}}$ is such that if some diagonal matrix commutes
with it, then it commutes with any other matrix satisfying all above
properties. Thus the matrix  $A_{\sigma_1^{(i)}}$ has the same
properties as $S_{\sigma_1^{(i)}}$. Consequently, as above (in
induction basis) we can conclude that every block has $a=0$.
Therefore on the block $\mathbf{D_i}$ we have
$A_{\sigma_1^{(i)}}=bS_{\sigma_1^{(i)}}$, $b^2=1$, what we needed.

It is clear that the proof for $S_{\sigma_m^{(i)}}$, $m> 1$, is
completely the same.
\end{proof}

\begin{definition}
\emph{By $\tau (i,p, m)$, $i=1,\dots ,l$, $p=1,\dots, k_i$,
$m=p,\dots,k_i$, we denote a substitution, that on all blocks,
except $\mathbf{D_i}$, is identical, and on the block $\mathbf{D_i}$
on the first $2^m$ basis elements coincides with $\sigma_p$, and on
other basis elements is identical. Therefore, $\tau (i,1,1)$ is just
a transposition, $\tau (i,p,k_i)=\sigma_p^{(i)}$.}
\end{definition}

\begin{lemma}\label{transp1}
Let an automorphism $\Phi$ of the semigroup $G_n(R)$ be such that
$\Phi(S_{\sigma_i})=a_iS_{\sigma_i}$, $a_i^2=1$, $i=1,\dots, k_1$.
Then $\exists M\in \Gamma_n(R)$ such that every matrix $S_{\tau
(i,1,m)}$, $i=1,\dots ,l$,  $m=1,\dots,k_i$, under the action
of~$\Phi_M\circ \Phi$ is mapped into $D\cdot S_{\tau(i,1,m)}$, where
$D\in{\mathcal Q}$.

If $\tau(i,1,m)$ is even, then $D=E$.
\end{lemma}

\begin{proof}
Let us fix some block $\mathbf{D_i}$ and prove the statement for all
$\tau(i,1,m)$, $m=1,\dots, k_i$ by induction from  $k_i$ to~$1$.
Note that for $k_i$ everything is already proved. Let us make the
step from  $k_i$ to~$k_i-1$.

Every matrix $S_{\tau(i,p,m)}$ commutes with any diagonal matrix,
that commutes with $S_{\tau(i,p,k_i)}$. It gives that for every
$m=p,\dots,k_i$ the image $A_{\tau(i,p,m)}=\Phi(S_{\tau(i,p,m)})$ is
diagonal in all blocks, except  $\mathbf{D_i}$, and since all
$S_{\tau(i,p,m)}$ commute with all $S_{\sigma_q^{(p)}}$, $q\ne i$,
we directly obtain that this image in all blocks is scalar. So we
only need to consider the block $\mathbf{D_i}$.

In this block we will use induction.

The matrix $A_{\tau(i,1,k_i-1)}$ commutes with all $S_{\sigma_q}$,
$q=1,\dots,k_i-1$, therefore on $\mathbf{D_i}$ the matrix
$A_{\tau(i,1,k_i-1)}$ consists of diagonal blocks $2\times 2$, and
the first half of blocks coincide with each other, the second half
also coincide with each other. Let every block in the first half of
blocks is equal to $\begin{pmatrix} a& b\\ b& a\end{pmatrix}$, every
block in the second half is equal to $\begin{pmatrix} c& d\\
d& c\end{pmatrix}$. Now use the condition
$$
S_{\tau(i,1,k_i-1)} S_{\sigma_{k_i}} S_{\tau(i,1,k_i-1)}
S_{\sigma_{k_i}} =S_{\tau(i,1,k_i)}.
$$
It gives
$$
\begin{pmatrix}
a& b\\
b& a
\end{pmatrix}
\begin{pmatrix}
c& d\\
d& c
\end{pmatrix}=\begin{pmatrix}
0& \alpha\\
\alpha& 0
\end{pmatrix},
$$
therefore $ac=bd=0$, $ad+bc=\alpha$. Changing basis with the matrix
$b^2E+a^2S_{\tau(i,k_i,k_i)}$ (this basis change commutes with all
introduced earlier matrices), we come to the matrix
$A_{\tau(i,1,k_i-1)}$, that has in the first half of $\mathbf{D_i}$ the matrices $\begin{pmatrix} 0& \alpha\\
\alpha& 0\end{pmatrix}$, $\alpha^2=1$, and in the second half the
matrices $\beta I_2$, $\beta^2=1$. Now let us make the basis change
(only in the block $\mathbf{D_i}$) with the help of diagonal matrix
$diag[\alpha/\beta,1,\alpha/\beta,1\dots, \alpha/\beta, 1]$. Such
change does not move $A_{\sigma_q^{(i)}}$, $q> 1$, and the matrix
$A_{\tau(i,1,k_i-1)}$ under consideration has in the new basis the
obtained form.

Again from the condition
$$
A_{\tau(i,1,k_i-1)} A_{\sigma_{k_i}^{(i)}} A_{\tau(i,1,k_i-1)}
A_{\sigma_{k_i}^{(i)}} =A_{\tau(i,1,k_i)}
$$
we see that $A_{\tau(i,1,k_i)}=S_{\tau(i,1,k_i)}$.

If  $k_i-1=1$, we do not need any later considerations. Therefore we
suppose that  $k_i$ is greater than~$2$. Let us show one more step
of induction (other steps are completely similar).

At the beginning we want to show that if in the block $\mathbf{D_i}$
$A_{\tau(i,1,k_i-1)}=\alpha S_{\tau(i,1,k_i-1)}$, then necessarily
$\alpha=1$.

Note that there exists a matrix $B$ (actually it is, for example,
the matrix $S_{(1,3)}$), that commutes with all diagonal matrices,
commuting with $S_{\sigma_2}$, and also a matrix $C$ (for example,
$S_{(1,4)}$), that commutes with all diagonal matrices, commuting
with $S_{\sigma_1}S_{\sigma_2}$, such that
$$
B\cdot S_{\tau(i,1,k_i-1)} \cdot B^{-1}\cdot
S_{\tau(i,1,k_i-1)}=C\cdot S_{\tau(i,1,k_i-1)} \cdot C^{-1}.
$$

Naturally, this condition  remains under any automorphism. Since the
matrix $\Phi(B)$ commutes with all diagonal matrices, commuting with
$S_{\sigma_2}$, then it independently acts in the first and in the
second halves of the block $\mathbf{D_i}$. The same we can say about
$\Phi(C)$. Since on the second half of the block $\mathbf{D_i}$ the
matrix $A_{\tau(i,1,k_i-1)}$ is scalar (has the form $\alpha I$),
then under conjugation it does not change its form. Therefore,
$\alpha^2=\alpha$, so $\alpha=1$.

Consequently, $A_{\tau(i,1,k_i-1)}=S_{\tau(i,1,k_i-1)}$.

 Now consider the matrix
$A_{\tau(i,1,k_i-2)}=\Phi(S_{\tau(i,1,k_i-2)})$.

We know that this matrix is scalar in the blocks $\mathbf{D_j}$,
$j\ne i$. Since it commutes with all diagonal matrices, commuting
with $A_{\tau(i,1,k_i-1)}$, we have that it is diagonal on the
second half of $\mathbf{D_i}$. Let us denote the first half of
$\mathbf{D_i}$ by $\mathbf{D_i'}$, the second half by
$\mathbf{D_i''}$.

The matrix $A_{\tau(i,1,k_i-2)}$ commutes with all $S_{\sigma_q}$,
$q=1,\dots,k_i-2$, therefore: 1) on $\mathbf{D_i'}$ it consists of
diagonal blocks $2\times 2$, where the first half of block are the
same, and similarly the second half of blocks coincide with each
other; 2) on $\mathbf{D_i''}$ it has the form
$diag[aI_{2^{k_i-2}},bI_{2^{k_i-2}}]$, $a^2=b^2=1$. Let on
$\mathbf{D_i}$ $A_{\tau(i,1,k_i-1)}=\alpha S_{\tau(i,1,k_i-1)}$.

Consider the condition
$$
A_{\tau(i,1,k_i-2)} A_{\sigma_{k_i-1}^{(i)}} A_{\tau(i,1,k_i-2)}
A_{\sigma_{k_i-1}^{(i)}} =S_{\tau(i,1,k_i-1)}.
$$
Repeat arguments, similar to the previous step, and take a basis
change with some suitable matrix
${b'}^2E+{a'}^2S_{\tau(i,k_i-1,k_i-1)}$. Such a change does not move
any $A_{\sigma_q^{(i)}}$, $q> 1$, and the matrix
$A_{\tau(i,1,k_i-2)}$ under this change has now the obtained form.

Continuing this procedure for $k_i-3,\dots,1$, we come to a basis
where all $A_{\tau(i,1,q)}$, $q>1$, coincide with $S_{\tau(i,1,q)}$,
and $A_{\tau(i,1,1)}$ differs from $S_{\tau(i,1,1)}$ by some element
from $\mathcal{Q}$.

Note that we proved also the last assertion of our lemma, because
odd elements under consideration are exactly $S_{\tau(i,1,1)}$.
\end{proof}

\begin{lemma}\label{transp2}
Let an automorphism $\Phi$ of $G_n(R)$ be such that every matrix
$S_{\tau (i,1,m)}$, $i=1,\dots ,l$, $m=1,\dots,k_i$, under the
action of~$\Phi$ is mapped to $D\cdot S_{\tau(i,1,m)}$, where
$D\in{\mathcal Q}$, and if a substitution  $\tau(i,p,m)$ is even,
then $D=I$. Then $\exists M\in \Gamma_n(R)$ such that every matrix
of the substitution~$\tau$, acting independently on the blocks
$\mathbf{D_1}$,\dots, $\mathbf{D_l}$, is mapped under $\Phi_M\circ
\Phi$ into a matrix $D S_\tau$, and if $\tau$ is even, then $D=I$.
\end{lemma}

\begin{proof}
Let us consider matrices of substitutions that act identically on
all blocks except some fixed block $\mathbf{D_i}$. Their images
commute with all diagonal matrices, commuting with all images of
$S_{\sigma_1^{(i)}}$,\dots, $S_{\sigma_{k_i}^{(i)}}$, and also with
images of all $S_{\sigma_j^{(p)}}$, $p\ne i$. Consequently in all
blocks $\mathbf{D_p}$, $p\ne i$, images of such matrices are scalar.
Consider now the block $\mathbf{D_i}$.

Note that if a size of $\mathbf{D_i}$ is not greater than two, then
everything is proved. Therefore we can suppose that this size is not
smaller than $4\times 4$. For convenience we will suppose that
elements of the basis in $\mathbf{D_i}$ are numbered from~$1$
(i.\,e. it is the first block).

By the condition of the lemma everything is proved for the matrix of
$(1,2)$, i.\,e. $\Phi(S_{(1,2)})=\alpha S_{(1,2)}$, $\alpha^2=1$.
Since the transpositions $(3,4)$, $(5,6)$,\dots,
$(2^{k_i}-1,2^{k_i})$ are conjugate to $(1,2)$ by substitutions
$\sigma_2$, $\sigma_3$,\dots, $\sigma_{k_i}$ and  their products,
then $\Phi(S_{(2p-1,2p)})=\alpha S_{(2p-1,22)}$ for all $p=2,\dots,
2^{k_i-1}$.

Consider the matrix $A_{(1,3)}$, the image of $S_{(1,3)}$. Since it
commutes with all diagonal matrices,  commuting with $S_{\sigma_2}$,
and also it commutes with $S_{\sigma_2}$ itself, we have that it is
divided into diagonal blocks $4\times 4$ of the form
$$\begin{pmatrix}
a& 0& b& 0\\
0& c& 0& d\\
b& 0& a& 0\\
0& d& 0& c
\end{pmatrix},\quad a^2+b^2=c^2+d^2=1, ab=cd=0.
$$
Since the matrix $A_{(1,2)}A_{(1,3)}$ has the order three, we have
that all blocks except the first one, are scalar matrices with the
coefficient~$\alpha$. Consider the first block. From the same
condition we obtain $ac=bd=0$, $a=d$, $b=c$. Make a new basis change
with the matrix $b^2I+a^2S_{\sigma_1^{(i)}}$. Such basis change does
not move any matrices introduced earlier, and our matrix now has the
form $\alpha S_{(1,3)}$. It is clear that all other matrices of
transpositions can be obtained by conjugation of matrices considered
above.
\end{proof}

\begin{lemma}\label{glavnygemorroy}
Let $n>4$, $\Phi$ be an automorphism of $G_n(R)$. Then $\exists M\in
\Gamma_n(R)$ and an involution $a\in R^*_+$, such that
$\Phi_1(S_\sigma)=\Phi_M \circ \Phi(S_\sigma)=S_\sigma$ for every
even substitution $\sigma$ and $\Phi_1(S_\sigma)=\Phi_M \circ
\Phi(S_\sigma)=aS_\sigma$ for every odd substitution~$\sigma$.
\end{lemma}

\begin{proof}
According to previous lemmas we can suppose that
$\Phi(S_{(p,p+1)})=D_pS_{(p,p+1)}$, $D_p\in {\mathcal Q}$,  for all
transpositions $(p,p+1)$, changing elements inside one block
$\mathbf{D_i}$, $i=1,\dots, l$. Since the product of any to
transpositions is even, and the matrix of an even substitution is
mapped to itself, then all matrices $D_p$ coincide. Since all
transpositions are conjugate, then the corresponding matrices  have
the same trace. If $D_p=D=diag[\alpha_1 I_{2^{k_1}},\dots, \alpha_l
I_{2^{k_l}}]$, then we consider to different blocks $\mathbf{D_i}$
and $\mathbf{D_j}$ of nonunit sizes. Consider transpositions
$\tau_i$ and $\tau_j$, arbitrary transpositions from corresponding
blocks. Trace of $A_{\tau_i}$ is $\alpha_1
2^{k_1}+\dots+\alpha_l2^{k_l}-2\alpha_i$, trace of $A_{\tau_j}$ is
$\alpha_1 2^{k_1}+\dots+\alpha_l2^{k_l}-2\alpha_j$. Therefore
$\alpha_i=\alpha_j$ always, except, possibly, a block $1\times 1$.

Let $n$ be odd (then there is a block $1\times 1$). Consider
$A_{(n-1,n)}=\Phi(S_{(n-1,n)})$.

We need to consider two cases: there exist a block $2\times 2$,
there is no such block.

{\bf 1.} Let us have a block $2\times 2$. It is staying just before
a block $1\times 1$.

The matrix $A_{(n-,n)}$ commutes with all transpositions from the
blocks $\mathbf{D_1}$,\dots, $\mathbf{D_{l-2}}$, and has the
order~$2$. Let us look what does it mean for the transposition
$S_{(1,2)}$ (by our assumption $n> 4$, i.\,e. $n>6$, therefore the
transposition $(1,2)$ does not intersect with $(n-2,n-1)$).

If
$$
A_{(n-1,n)}=\begin{pmatrix} a_{11}& a_{12}& a_{13}&\dots & a_{1n}\\
a_{21}& a_{22}& a_{23}& \dots& a_{2n}\\
a_{31}& a_{32}& a_{33}& \dots& a_{3n}\\
\vdots& \vdots& \vdots& \ddots& \vdots\\
a_{n1}& a_{n2}& a_{n3}&\dots& a_{nn}
\end{pmatrix},
$$
then since $A_{(n-1,n)}$ commutes with $S_{(1,2)}$, we have
$a_{2i}=a_{1i}$, $a_{i2}=a_{i1}$, $i=3,\dots,n$, $a_{22}=a_{11}$,
$a_{12}=a_{21}$. From the other side the matrix $A_{(n-1,n)}$ has
the order two, therefore
$a_{11}^2+a_{12}^2+a_{13}a_{31}+\dots+a_{1n}a_{n1}=1$ (the first
column and the first row), $a_{11}a_{12}=a_{13}a_{31}=\dots
=a_{1n}a_{n1}=0$ (the second column and the first row). As usual for
matrices of order two $a_{ik}a_{kj}=0$ for $i\ne j$. Consequently,
$a_{11}^2+a_{12}^2=1$, and since
$a_{1i}a_{11}=a_{i1}a_{11}=a_{1i}a_{12}=a_{i1}a_{12}=0$, $i=3,\dots,
n$, we have $a_{1i}=a_{i1}=0$ for all $i=3,\dots,n$.

By similar arguments for other transpositions commuting with
$A_{(n-1,n)}$, we see that $A_{(n-1,n)}$ in every block
$\mathbf{D_i}$, $i=1,\dots, l-2$, is scalar and in the end it has
the block $3\times 3$. It is clear that we can bound our
consideration on this block. But in this case we come to the
situation of Lemma~\ref{n3}, and by a suitable basis change (of the
form $diag[1,1,\dots,1,\gamma](aE+bS_{(n-2,n-1)})$, i.\,e. commuting
with all matrices considered above) we can come to the situation
$A_{(n-1,n)}=DS_{(n-1,n)}$, $D\in {\mathcal Q}$, $D$ is scalar on
the union of last two blocks. Since the matrix
$A_{(n-2,n-1)}A_{(n-1,n)}$ has order three, we obtain directly
$D=\alpha E$. It is what we needed.

{\bf 2.} Suppose now that there is no block $2\times 2$. This case
is even easier than the previous one. Initial arguments are similar
to the previous arguments, according to them we obtain that
$A_{(n-1,n)}$ in every block $\mathbf{D_i}$, $i=1,\dots, l-2$, is
scalar. Then let us consider the block that is previous to the last
one, it has size at least $4\times 4$. Since $A_{(n,n-1)}$ commutes
with all matrices of transpositions $(p,p+1)$ in this block, except
the last one, we obtain that the matrix $A_{(n-1,n)}$ is scalar on
elements of the block, except the last one, i.,e. the matrix
$A_{(n-1,n)}$ is diagonal everywhere, except the last block $2\times
2$. Then it is clear that it has the form $D S_{(n-1,n)}$ for some
diagonal matrix $D$, and after a basis change with the matrix
$diag[1,1,\dots,1,\gamma]$ we come to $A_{(n-1,n)}=\alpha
S_{(n-1,n)}$.

Completely the same arguments can be applied  to transpositions,
that join other pairs of blocks. On every step (jointing blocks
$\mathbf{D_i}$ and $\mathbf{D_{i+1}}$) it is sufficient to apply
diagonal changes, that are identical on the blocks
$\mathbf{D_1},\dots, \mathbf{D_i}$, and scalar with some
coefficient~$\gamma$ on other blocks. Clear that such changes
commute with all considered above matrices of transpositions, what
we need.

Finally we obtain  $A_{(p,p+1)}=\alpha S_{(p,p+1)}$ for all
$p=1,\dots, n-1$, so the lemma is proved.
\end{proof}

\section{Action of $\Phi'$ on diagonal matrices.}\leavevmode

In the previous section by our initial automorphism $\Phi$ we
constructed some new automorphism  $\Phi'=\Phi_M\Phi$ such that
$\Phi' (S_\sigma)=\alpha^{sgn\,\sigma}S_\sigma$, $\alpha^2=1$, for
all $\sigma\in \Sigma_n$. We suppose that such an automorphism
$\Phi'$ is fixed.

\begin{lemma}\label{d1}
If $n\ge 3$, $1/2\in R$, an automorphism $\Phi' \in Aut(G_n(R))$ is
such that $\forall \sigma\in \Sigma_n$
$\Phi'(S_\sigma)=\alpha^{sgn\,\sigma}S_\sigma$, $\alpha^2=1$, then
for all $\alpha,\beta\in R_+^*$ we have
$$
\Phi'
(diag[\alpha,\beta,\dots,\beta])=diag[\gamma,\delta,\dots,\delta],\quad
\gamma,\delta\in R_+^*.
$$
If $\alpha\ne \beta$, then $\gamma\ne \delta$.

\end{lemma}

\begin{proof}
By Lemma~\ref{diag}
$$
\Phi'(diag [\alpha,\beta,\dots,\beta])=diag
[\gamma_1,\dots,\gamma_n].
$$
Consider the substitution $\sigma =(2,3,\dots,n)$. Since
$\Phi'(S_{\sigma})=\alpha S_{\sigma}$ then
$\Phi'(C_{D_n(R)}(S_\sigma)=C_{D_n(R)}(S_\sigma)$ Therefore $diag
[\gamma_1,\dots,\gamma_n]$ commutes with  $S_\sigma$, and therefore
 $\gamma_2=\gamma_3=\dots=\gamma_n$. So  we only need to prove that
$\gamma_1\ne\gamma_2$. It directly follows from the fact that the
matrix $diag[\alpha,\beta,\dots,\beta]$ does not commute with
$S_{(12)}$.
\end{proof}

\begin{lemma}\label{d2}
If $n\ge 3$, $1/2\in R$, an automorphism $\Phi'\in Aut(G_n(R)$ is
such that
 $\forall \sigma\in \Sigma_n$ $\Phi'(S_\sigma)=\alpha^{sgn\,\sigma}S_\sigma$, $\alpha^2=1$, then for all
 $X\in G_2(R)$ we have
$$
\Phi'\begin{pmatrix}
X& 0& \dots& 0\\
0& 1 & \dots& 0\\
\hdotsfor{2}& \ddots& \dots\\
0& \hdotsfor{2}& 1
\end{pmatrix}
=\begin{pmatrix}
Y& 0& \dots& 0\\
0& a & \dots& 0\\
\hdotsfor{2}& \ddots& \dots\\
0& \hdotsfor{2}& a
\end{pmatrix},\quad \text{ where } Y\in G_2(R), a\in R_+^{*}.
$$
\end{lemma}
\begin{proof}
Denote $$C=\begin{pmatrix}
X& 0& \dots& 0\\
0& 1 & \dots& 0\\
\hdotsfor{2}& \ddots& \dots\\
0& \hdotsfor{2}& 1
\end{pmatrix}
$$
Similarly to the proof of Lemma~\ref{d1} we can prove that for any
matrix
$$
A=diag [\alpha,\alpha,\beta,\dots,\beta]\in D_n(R),\quad \alpha\ne
\beta,
$$
we  have
$$
\Phi'(A)=diag [\gamma,\gamma,\delta,\dots,\delta]\in D_n(R),\quad
\gamma\ne \delta.
$$
Note that the matrix  $C$ commutes with $A$ and $S_{(3,4,\dots,n)}$,
consequently  $\Phi'(C)$ has the obtained form.

\end{proof}

\begin{lemma}\label{d3}
If $n\ge 3$, $1/2\in R$, an automorphism $\Phi'\in Aut(G_n(R))$ is
such that $\forall \sigma\in \Sigma_n$
$\Phi'(S_\sigma)=\alpha^{sgn\,\sigma}S_\sigma$, $\alpha^2=1$, then
for any
 $x_1,x_2\in R_+^*$  with
$x_1\ne x_2$,
\begin{align*}
\Phi'(A_1)&= \Phi'(diag [x_1,1,\dots,1])=diag [\xi_1,\eta_1,\dots,\eta_1],\\
\Phi'(A_2)&= \Phi'(diag [x_2,1,\dots,1])= diag
[\xi_2,\eta_2,\dots,\eta_2]
\end{align*}
we have $\xi_1\eta_1^{-1}\ne \xi_2\eta_2^{-1}$.
\end{lemma}
\begin{proof}
Suppose that for some different $x_1,x_2\in R_+^*$ we have
$\xi_1\eta_1^{-1}= \xi_2\eta_2^{-1}$, i.\,e.
\begin{align*}
\Phi'(A_1)&= \Phi'(diag [x_1,1,\dots,1])=diag [\xi,\eta,\dots,\eta]=A_1',\\
\Phi'(A_2)&= \Phi'(diag [x_2,1,\dots,1])=\alpha\cdot
 diag [\xi,\eta,\dots,\eta]=A_2',
\end{align*}
 Therefore, ${\Phi'}^{-1}
(\alpha I)={\Phi'}^{-1}(A_1' {A_2'}^{-1})=
diag[x_1x_2^{-1},1,\dots,1])=diag[\beta,1,\dots,1]$, where $1\ne
\beta\in R_+^* $, but it is impossible.  Consequently,
$\xi_1\eta_1^{-1}\ne \xi_2\eta_2^{-1}$.
\end{proof}

\section{Main theorem.}\leavevmode

In this section we will prove the main theorem (Theorem~1).

\begin{lemma}\label{otobr}
If $n\ge 3$, $1/2\in R$, an automorphism $\Phi'\in Aut(G_n(R))$ is
such that $\forall \sigma\in \Sigma_n$
$\Phi'(S_\sigma)=\alpha^{sgn\,\sigma}S_\sigma$, $\alpha^2=1$, then
there exists such a mapping $c(\cdot ): R_+\to R_+$ that for all
$x\in R_+$ $\Phi'(B_{12}(x))=B_{12}(c(x))$.
\end{lemma}

\begin{proof}
By Lemma~\ref{d2} we have
$$
\Phi'(B_{12}(1))=\begin{pmatrix}
\alpha& \beta&&&\\
\gamma& \delta&&&\\
&&a&&\\
&&&\ddots&\\
&&&&a
\end{pmatrix}, \quad a\in R_+^*,
\begin{pmatrix}
\alpha& \beta\\
\gamma&\delta
\end{pmatrix} \in G_2(R).
$$

Let for every $x\in R_+^*$
$$
\Phi'(diag[x,1,\dots,1)]=diag[\xi(x),\gamma(x),\dots,\gamma(x)],\quad
\xi(x),\eta(x)\in R_+^*
$$
(see Lemma~\ref{d1}).

Then for any $x\in R^*_+$
\begin{multline*}
\Phi' (B_{12}(x))=\Phi' (diag[x,1,\dots,1]
B_{12}(1) diag[x^{-1},1,\dots,1])=\\
=diag [\xi(x),\eta(x),\dots,\eta(x)]
\begin{pmatrix}
\alpha& \beta& &&\\
\gamma& \delta&&&\\
&&a&&\\
&&&\ddots&\\
&&&&a
\end{pmatrix}
diag[\xi(x)^{-1},\eta(x)^{-1},\dots ,\eta(x)^{-1}]=\\
=\begin{pmatrix}
\alpha& \nu(x)\beta& &&\\
\nu(x)^{-1}\gamma& \delta&&&\\
&&a&&\\
&&&\ddots&\\
&&&&a
\end{pmatrix}
\end{multline*}
 with $\nu(x)=\xi(x)\eta(x)^{-1}$.

By Lemma~\ref{d3} for $x_1\ne x_2$ we have $\nu(x_1)\ne \nu(x_2)$.

For every $x\in R_+$ $\Phi'(B_{12}(1))$ and $\Phi' (B_{12}(x))$
commute. Let us write this assertion in the matrix form for $x\in
R_+^*$:
\begin{multline*}
\begin{pmatrix}
\alpha& \beta\\
\gamma& \delta
\end{pmatrix}
\begin{pmatrix}
\alpha& \nu(x)\beta\\
\nu(x)^{-1}\gamma & \delta
\end{pmatrix} =\begin{pmatrix} \alpha& \nu(x)\beta\\
\nu(x)^{-1} \gamma& \delta
\end{pmatrix}
\begin{pmatrix}
\alpha& \beta\\
\gamma& \delta
\end{pmatrix} \Rightarrow\\
\Rightarrow
\begin{pmatrix}
\alpha^2+\nu(x)^{-1} \beta\gamma& \nu(x)\alpha\beta+\beta\delta\\
\gamma\alpha+\nu(x)^{-1}\delta\gamma& \nu(x)\gamma \beta+\delta^2
\end{pmatrix}=
\begin{pmatrix}
\alpha^2+\nu(x) \beta\gamma& \alpha\beta+\nu(x)\beta\delta\\
\nu(x)^{-1}\gamma\alpha+\delta\gamma& \nu(x)^{-1}\gamma
\beta+\delta^2
\end{pmatrix}.
\end{multline*}
Therefore, $\nu(x)^{-1}\beta\gamma = \nu(x)\beta\gamma$ for distinct
$x\in R_+^*$. Hence, $\beta\gamma=0$.

Let us use the
condition
$(B_{12}(1))^2=diag[2,1,\dots,1]B_{12}(1)\cdot diag [1/2,
1,\dots,1]$:
$$
\begin{pmatrix}
\alpha^2& \beta(\alpha+\delta)& &&\\
\gamma (\alpha+\delta)& \delta^2&&&\\
&&a^2&&\\
&&&\ddots&\\
&&&&a^2
\end{pmatrix}=
\begin{pmatrix}
\alpha& \nu(2)\beta& &&\\
\nu(2)^{-1}\gamma& \delta&&&\\
&&a&&\\
&&&\ddots&\\
&&&&a
\end{pmatrix},
$$
that implies $\alpha^2=\alpha$, $\delta^2=\delta$, $a=1$ (since $a$
is invertible).

Use the condition $B_{12}(1)B_{13}(1)=B_{13}(1)B_{12}(1)$, it
implies
$$
\begin{pmatrix}
\alpha^2& \beta& \alpha\beta\\
\alpha\gamma& \delta& \gamma\beta\\
\gamma& 0& \delta
\end{pmatrix}=\begin{pmatrix}
\alpha^2&\beta\alpha&\beta\\
\gamma& \delta& 0\\
\alpha\gamma& 0 & \delta \end{pmatrix},
$$
therefore $\alpha\gamma=\gamma$.

Now we can use the condition
$B_{12}(1)B_{13}(1)=B_{13}(1)B_{23}(1)B_{12}(1)$, it gives
$$
\begin{pmatrix}
\alpha&\alpha\beta& \beta^2\\
\gamma& \alpha\delta& \beta\delta\\
0& \gamma& \delta
\end{pmatrix}=\begin{pmatrix}
\alpha^2& \alpha\beta& \beta\delta\\
\gamma&\alpha\delta& \beta\\
\alpha\gamma+\gamma^2\delta& \gamma\delta^2 &\delta^2
\end{pmatrix}.
$$
Thus, $\alpha\gamma+\gamma^2\delta=0$, consequently
$\alpha\gamma=\gamma=0$. Since $\gamma=0$, we have that elements
$\alpha$ and $\delta$ are invertible, since $\alpha$ and $\delta$
are idempotents, then $\alpha=\delta=1$. Therefore,
$\Phi'(B_{12}(1))=B_{12}(\beta)$.

Let us consider $B_{12}(x)$, $x\in R_+$. From the condition
$B_{12}(1)B_{23}(x)=B_{13}(x)B_{23}(x)B_{12}(1)$ it follows
$$
\begin{pmatrix}
1& a\beta& b\beta\\
0& a& b\\
0& c& d
\end{pmatrix}=\begin{pmatrix}
a& a\beta+bc& bd\\
0& a& b\\
c& c\beta+cd& d^2
\end{pmatrix}.
$$
So $c=0$, therefore $a$ and $d$ are invertible. From the condition
$B_{12}(x)^2=diag[2,1,\dots,1]^{-1} B_{12}(x) diag [2,1,\dots,1]$ it
follows $a^2=a$, $d^2=d$. Hence $a=d=1$. Thus,
$\Phi'(B_{12}(x))=B_{12}(b(x))$.
\end{proof}

Recall (see Definition~8), that if $G$ is a semigroup, then  a
homomorphism $\lambda(\cdot ): G\to G$ is called a \emph{ central
homomorphism} of~$G$, if $\lambda(G)\subset Z(G)$. A mapping
$\Omega(\cdot): G\to G$ such that $\forall X\in G$
$$
\Omega (X)=\lambda(X)\cdot X,
$$
where $\lambda(\cdot)$ is a central homomorphism, is called a
\emph{central homothety}.

Recall that for every $y(\cdot)\in Aut (R_+)$ by $\Phi^y$ we denote
an automorphism of $G_n(R)$ such that $\forall X=(x_{ij})\in G_n(R)$
$\Phi^y(X)=\Phi^y((x_{ij}))=(y(x_{ij}))$.

\begin{theorem}
Suppose that $\Phi$ is an arbitrary automorpism of $G_n(R)$, $n\ge
3$, $1/2\in R$. Then on the semigroup $GE_n^+(R)$ \emph{(}see
Definition~\emph{7)} $\Phi=\Phi_M\Phi^c\Omega$, where $M\in
\Gamma_n(R)$, $c(\cdot)\in Aut(R_+)$, $\Omega(\cdot)$ is a central
homothety of $GE_n^+(R)$.
\end{theorem}
\begin{proof}
By Lemmas \ref{n3}, \ref{n4}, \ref{glavnygemorroy} there exists such
a matrix $M'\in \Gamma_n(R)$, that for every substitution $\sigma\in
\Sigma_n$
$$
\Phi'(S_\sigma)=\Phi_{M'}\Phi(S_\sigma)=\alpha^{sgn\,\sigma}S_\sigma,
\quad \alpha^2=1.
$$

Now let us consider the automorphism~$\Phi'$.

Be Lemma~\ref{otobr} there exists a mapping $b(\cdot): R_+\to R_+$
such that for any $x\in R_+$
$$
\Phi'(B_{12}(x))=B_{12}(b(x)).
$$

Consider this mapping. Since $\Phi'$ is an automorphism of $G_n(R)$,
we have that $b(\cdot):R_+\to R_+$ is bijective.

Since for all $x_1,x_2\in R_+$\
$B_{12}(x_1+x_2)=B_{12}(x_1)B_{12}(x_2)$, then
\begin{multline*}
B_{12}(b(x_1+x_2))=\Phi'(B_{12}(x_1+x_2))=\Phi'(B_{12}(x_1)
B_{12}(x_2))=\\
=\Phi'(B_{12}(x_1))\Phi'(B_{12}(x_2))=B_{12}(b(x_1))\cdot
B_{12}(b(x_2))=B_{12}(b(x_1)+b(x_2)), \end{multline*}
 therefore for all $x_1,x_2\in R_+$ $c( x_1+x_2)=c(x_1)+c(x_2)$, consequently
$b(\cdot)$ is additive.

To prove multiplicativity of $b(\cdot)$, we will use the condition:

1)
$\Phi'(B_{13}(x))=\Phi'(S_{(2,3)}B_{12}(x)S_{(2,3)})=S_{(2,3)}B_{12}
(b(x))S_{(2,3)}=B_{13}(b(x))$;

2) similarly, $\Phi'(B_{32}(x))=B_{32}(b(x))$;

3) (compare with the proof of Lemma~\ref{otobr})
\begin{multline*}
B_{13}(x_1)B_{32}(x_2)=B_{32}(x_2)B_{13}(x_1)B_{12}(x_1x_2)\Rightarrow\\
\Rightarrow
\Phi'(B_{13}(x_1))\Phi'(B_{32}(x_2))=\Phi'(B_{32}(x_2))\Phi'
(B_{13}(x_1))\Phi'(B_{12}(x_1x_2))\Rightarrow\\
\Rightarrow
B_{13}(b(x_1))B_{32}(b(x_2))=B_{32}(b(x_2))B_{13}(b(x_1))B_{12}
(b(x_1x_2))\Rightarrow\\
\Rightarrow \forall x_1,x_2\in R_+\
\begin{pmatrix} 1& b(x_1)b(x_2)& b(x_1)\\
0& 1& 0\\
0& b(x_2)& 1
\end{pmatrix}=\begin{pmatrix}
1& b(x_1x_2)& b(x_1)\\
0& 1& 0\\
0& b(x_2)& 1
\end{pmatrix}\Rightarrow\\
\Rightarrow \forall x_1,x_2\in R_+\ b(x_1x_2)=b(x_1)b(x_2).
\end{multline*}

Therefore $b(\cdot)$ is multiplicative.

Since $b(\cdot)$ is bijective, additive and multiplicative, then
$b(\cdot )$ is an automorphism of the semiring~$R_+$.

Consider now the mapping $\Phi^{b^{-1}}$, that maps every matrix
$A=(a_{ij})$ to $\Phi^{b^{-1}}(A)=(b^{-1}(a_{ij}))$. This mapping is
an automorphism of $G_n(R)$. Then $\Phi''=\Phi^{b^{-1}}\circ
\Phi'=\Phi^{b^{-1}}\circ \Phi_{M'} \circ \Phi$ is an automorphism of
$G_n(R)$, that does not move
 $B_{ij}(x)$ ($x\in R_+$,
$i,j=1,\dots, n$, $i\ne j$) and
$\Phi''(S_\sigma)=\alpha^{sgn\,\sigma}S_\sigma$ ($\sigma \in
\Sigma_n$). Namely,
 $\Phi''(S_\sigma)=\Phi^{b^{-1}}(\Phi'(S_\sigma))=\Phi^{b^{-1}}(\alpha S_\sigma)=
b^{-1}(\alpha) S_\sigma$, since the matrix $S_\sigma$ contains only
$0$ and~$1$; for $i=3,\dots,n$ $\Phi''
(B_{i2}(x))=\Phi''(S_{(1,i)}B_{12}(x)S_{(1,i)})=
S_{(1,i)}\Phi''(B_{12}(x)))S_{(1,i)}=S_{(1,i)}\Phi^{b^{-1}}(B_{12}(b(x)))
S_{(1,i)}=S_{(1,i)}B_{12}(x)S_{(1,i)}=B_{i,2}(x)$; for $j=3,\dots
,n$ $\Phi''(B_{1j}(x))=\Phi''(S_{(2,j)}
B_{12}(x)S_{(2,j)})=S_{(2,j)} B_{12}(x) S_{(2,j)}=B_{1j}(x)$; for
$i,j=3,\dots,n$
$\Phi''(B_{ij}(x))=\Phi''(S_{(i,1)}B_{1j}(x)S_{(1,i)})=S_{(1,i)}B_{1j}(x)
S_{(1,i)}=B_{ij}(x)$.

As we know (see Lemma~\ref{d1}), for all $\alpha\in R_+^*$
$$
\Phi''(diag [\alpha,1,\dots,1])=diag
[\beta(\alpha),\gamma(\alpha),\dots,\gamma (\alpha)],\quad
\beta,\gamma\in R_+^*.
$$
Use the condition
\begin{multline*}
diag[\alpha,1,\dots,1]B_{12}(1)diag[\alpha^{-1},1,\dots,1]=B_{12}(\alpha)\Rightarrow\\
\Phi''(diag[\alpha,1,\dots,1])\Phi''(B_{12}(1))\Phi''(diag
[\alpha^{-1}
,1,\dots,1])=\Phi''(B_{12}(\alpha))\Rightarrow\\
\Rightarrow
diag[\beta(\alpha),\gamma(\alpha),\dots,\gamma(\alpha)]B_{12}(1)diag
[\beta
(\alpha)^{-1},\gamma(\alpha)^{-1},\dots,\gamma(\alpha)^{-1}]=B_{12}(\alpha)\Rightarrow\\
\Rightarrow \beta(\alpha)\gamma(\alpha)^{-1}=\alpha\Rightarrow
\beta(\alpha)=\alpha
\gamma(\alpha)\Rightarrow\\
\forall \alpha\in R_+^*\ \Phi''(diag[\alpha,1,\dots,1])=diag [\alpha
\gamma(\alpha), \gamma(\alpha),\dots,\gamma(\alpha)].
\end{multline*}

Since for all $\alpha_1,\alpha_2\in R_+^*$
\begin{multline*}
diag[\alpha_1\alpha_2\gamma(\alpha_1\alpha_2),\gamma(\alpha_1\alpha_2),\dots,
\gamma(\alpha_1\alpha_2)]=\Phi''(diag [\alpha_1\alpha_2,1,\dots,1])=\\
=\Phi''(diag[\alpha_1,1,\dots,1])\Phi''(diag[\alpha_2,1,\dots,1])=\\
=diag[\alpha_1\gamma(\alpha_1),\gamma(\alpha_1),\dots,\gamma(\alpha_1)]diag
[\alpha_2\gamma(\alpha_2),\gamma(\alpha_2),\dots,\gamma(\alpha_2)]=\\
=diag
[\alpha_1\alpha_2\gamma(\alpha_1)\gamma(\alpha_2),\gamma(\alpha_1)\gamma(\alpha_2),
\dots,\gamma(\alpha_1)\gamma(\alpha_2)]\Rightarrow\\
\Rightarrow \forall \alpha_1,\alpha_2\in R_+^*\
\gamma(\alpha_1\alpha_2)=\gamma(\alpha_1)\gamma(\alpha_2),
\end{multline*}
then the mapping $\gamma(\cdot)$ is a central homomorphism (see
Definition~8) $\gamma(\cdot): R_+^*\to R_+^*$.

If $A=diag[\alpha_1,\dots,\alpha_n]\in D_n(R)$, then
\begin{multline*}
\Phi''(A)=\\
=\Phi''(diag[\alpha_1,1,\dots,1] S_{1,2}diag[\alpha_2,1,\dots,1]
S_{(1,2)}
S_{(1,3)} diag [\alpha_3,1,\dots,1]\times\\
\times S_{(1,3)} \dots S_{(1,n)} diag[\alpha_n,1,
\dots,1]S_{(1,n)})=\\
=\gamma(\alpha_1)diag[\alpha_1,1,\dots,1]S_{(1,2)}\gamma(\alpha_2)
diag
[\alpha_2,1,\dots,1]S_{(1,2)}\dots\\
\dots S_{(1,n)} \gamma(\alpha_n) diag[\alpha_n,1,\dots,1]
\gamma(\alpha_n)=\\
=\gamma(\alpha_1)\dots
\gamma(\alpha_n)A=\gamma(\alpha_1\dots\alpha_n)A.
\end{multline*}

Recall (see Definition~5), that $\mathbf P$ is a subsemigroup of
$G_n(R)$, generated by~$S_\sigma$ ($\sigma\in \Sigma_n$),
$B_{ij}(x)$ ($x\in R_+$, $i,j=1,\dots,n$, $i\ne j$), and $diag
[\alpha_1,\dots,\alpha_n]$ ($\alpha_1,\dots,\alpha_n\in R^*_+$).

Note that determinant of any matrix from $GE^+_n(R)$ is an
invertible element of $G_n(R)$, that can be compared with zero (it
is $\ge 0$ or $\le 0$). It follows from the fact that all diagonal
matrices have determinant $\ge 0$,  all matrices of substitutions
have determinant $\pm 1$, matrices $B_{ij}(x)$ and their inverse
matrices have determinant~$1$.

Let $\Phi''(S_{(1,2)})=\alpha S_{(1,2)}$, $\alpha^2=1$. Consider the
mapping $\mu: R_+^*\cup R_-^*\to R_+^*$, that corresponds every
$a\in R_+^*$ to itself, and every $a\in R_-^*$ to $\alpha a$. Clear
that it is a homomorphism.

Then the mapping that corresponds every matrix $A$ to the matrix
$\mu(\det A)A$, is a central homothety of $GE_n^+(R)$. Let us denote
this homothety by $\Omega'$ and consider the composition
$\Phi'''=\Omega'\circ \Phi''$. It is an automorphism of $GE_n^+(R)$,
that does not move $S_\sigma$, $\sigma\in \Sigma_n$, and
$B_{ij}(x)$, $i\ne j$, $x\in R_+$.

Clear that every matrix $A\in \mathbf P$ can be represented as
$$
A=diag[\alpha_1,\dots,\alpha_n]A_1\dots A_k,
$$
where $\alpha_1,\dots,\alpha_n\in R^*_+$, $A_1,\dots,A_k\in \{
S_\sigma, B_{ij}(x)| \sigma\in \Sigma_n, x\in R_+, i,j=1,\dots,n,
i\ne j\}$. Then
\begin{multline*}
\Phi'''(A)=\Phi'''(diag[\alpha_1,\dots,\alpha_n]A_1\dots A_k)=\\
=\gamma(\alpha_1 \dots \alpha_n)diag
[\alpha_1,\dots,\alpha_n]A_1\dots A_k= \gamma (\alpha_1\dots
\alpha_n)A.
\end{multline*}
Now we introduce a mapping $\overline \gamma(\cdot): {\mathbf P}\to
R_+^*$ by the following rule: if $A\in \mathbf P$ and
$A=diag[\alpha_1,\dots,\alpha_n]A_1\dots A_k$, where
$A_1,\dots,A_k\in \{ S_\sigma, B_{ij}(x)| \sigma\in \Sigma_n, x\in
R_+, i,j=1,\dots,n, i\ne j\}$, then $\overline \gamma
(A)=\gamma(\alpha_1,\dots,\alpha_n)$.

The mapping $\overline \lambda(\cdot)$ is uniquely defined, since if
$$ A=diag[\alpha_1,\dots,\alpha_n]A_1\dots
A_k=\\
=diag [\alpha_1', \dots, \alpha_n']A_1'\dots A_m',
$$
 then
 $\Phi'''(A)=\gamma(\alpha_1\dots \alpha_n)A$ and $\Phi'''(A)=\gamma(\alpha_1'
\dots \alpha_n')A$, therefore
 $\gamma(\alpha_1\dots \alpha_n)=\gamma(\alpha_1'\dots\alpha_n')$.

Since $\overline \gamma (AA')
AA'=\Phi'''(AA')=\Phi'''(A)\Phi'''(A')=\overline \gamma (A) A\cdot
\overline \gamma(A')A'=\overline \gamma(A) \overline \gamma(A')AA'$,
then $\overline \gamma$ is a homomorphism $\mathbf P\to R_+^*$.

Now we see that on~$\mathbf P$ the automorphism
 $\Phi'''$ coincides with a central homothety
 $\Omega(\cdot): \mathbf P\to \mathbf P$, where for all $a\in
\mathbf P$ $\Omega(A)=\overline \gamma(A)\cdot A$.

Let $B\in GE_n^+(R)$. Then (see Definitions~6,7) a matrix $B$ is
$\mathcal P$-equivalent to some matrix $A\in \mathbf P$, i.e. there
exist matrices
 $A_0,\dots, A_k\in G_n(R)$,
$A_0=A\in \mathbf P$, $A_k=B$ and matrices $P_i,\widetilde P_i, Q_i,
\widetilde Q_i\in \mathbf P$, $i=0,\dots,k-1$ such that for all
$i=0,\dots,k-1$
$$
P_i A_i \widetilde P_i=Q_i A_{i+1}\widetilde Q_i.
$$
Then
\begin{multline*}
\Phi'''(P_0A_0\widetilde P_0)=\Phi'''(Q_0A_1\widetilde Q_0)\Rightarrow\\
\Rightarrow \overline \gamma (P_0)P_0\overline \gamma
(A_0)A_0\overline \gamma (\widetilde P_0)\widetilde P_0=\overline
\gamma(Q_0)Q_0\Phi''(A_1)\overline \gamma(\widetilde Q_0)\widetilde
Q_0
\Rightarrow\\
\overline \gamma (P_0A_0\widetilde P_0)P_0A_0\widetilde
P_0=\overline \gamma (Q_0\widetilde
Q_0)Q_0 \Phi'''(A_1)\widetilde Q_0\Rightarrow\\
\Rightarrow \overline \gamma (P_0A_0\widetilde P_0)\overline
\gamma(Q_0\widetilde
Q_0)^{-1} Q_0A_1\widetilde Q_0=Q_0\Phi'''(A_1)\widetilde Q_0\Rightarrow\\
\Rightarrow \Phi'''(A_1) =\overline \gamma (P_0A_0\widetilde
P_0)\overline \gamma
(Q_0\widetilde Q_0)^{-1}A_1,\dots,\\
\dots, \Phi'''(B)=\Phi''(A_n)=\overline \gamma (P_{n-1})\overline
\gamma (A_{n-1}) \overline \gamma (\widetilde P_{n-1})\overline
\gamma (Q_{n-1})^{-1} \overline \gamma (\widetilde Q_{n-1}).
\end{multline*}
Consequently, we can extend the mapping
 $\overline \gamma(\cdot): \mathbf P\to R_+^*$
to some mapping $\lambda(\cdot): GE_n^+(R)\to R_+^*$ such that for
every $B\in GE_n^+(R)$
$$
\Phi'''(B)=\lambda (B)\cdot B.
$$

Since $\Phi'''$ is an automorphism of $GE_n^+(R)$, then
$\lambda(\cdot)$ is a central homomorphism
 $\lambda(\cdot): GE_n^+(R)\to R_+^*$ and, therefore,
automorphism $\Phi''':GE_n^+(R)\to GE_n^+(R)$ is a central homothety
$\Omega''(\cdot): GE_n^+(R)\to GE_n^+(R)$, where
 $\forall X\in GE_n^+(R)$ $\Omega''(X)=\lambda
(X)\cdot X$.

Since $\Phi'''=\Omega''$ on $GE_n^+(R)$ and
 $\Phi'''=\Omega'\circ \Phi^{c^{-1}}\circ \Phi_{M'}\circ \Phi$
on $G_n(R)$, then $\Phi=\Phi_M\circ \Phi^c \circ \Omega$ on
$GE_n^+(R)$, where $M={M'}^{-1}$, $\Omega=\Omega'\Omega''$.
\end{proof}

\end{document}